\documentclass[a4paper,leqno,11pt,twoside]{amsart}

\usepackage{amsthm,amsfonts,amsbsy,amssymb,amsmath}
\usepackage[all]{xy}
\usepackage{mathrsfs}  
\usepackage{extarrows}
\usepackage{array}
\usepackage{enumitem} 
\usepackage{graphicx}
\usepackage{subfigure}
\usepackage{afterpage}

\usepackage[bookmarks, breaklinks, pagebackref, colorlinks, 
pdftitle={Bayer-Macri decomposition on Bridgeland moduli spaces over surfaces},
pdfsubject={Math, Algebraic Geometry},
pdfauthor={Wanmin Liu},
pdfkeywords={Bayer-Macri map, Bridgeland stability conditions, minimal model program, moduli spaces of complexes, wall-crossing},
]{hyperref}

\theoremstyle{plain}
\newtheorem{theorem}{Theorem}[section]

\newtheorem{corollary}[theorem]{Corollary}

\newtheorem{lemma}[theorem]{Lemma}

\theoremstyle{definition}
\newtheorem{assumption}[theorem]{Assumption}

\newtheorem{definition}[theorem]{Definition}
\newtheorem{example}[theorem]{Example}

\newtheorem{notation}[theorem]{Notation}
\newtheorem{proposition-definition}[theorem]{Proposition-Definition}

\theoremstyle{remark}
\newtheorem{remark}[theorem]{Remark}

\numberwithin{equation}{section}

\def\into{\ensuremath{\hookrightarrow}}

\newcommand{\Db}{\mathrm{D}^{\mathrm{b}}} 
\newcommand{\DX}{\Db(X)} 
\newcommand{\DS}{\Db(S)} 
\newcommand{\KX}{\mathrm{K}_X} 
\newcommand{\KS}{\mathrm{K}_S} 

\newcommand{\T}{\mathcal{T}} 
\newcommand{\F}{\mathcal{F}} 
\newcommand{\OO}{\mathcal{O}} 
\newcommand{\A}{\mathcal{A}} 

\newcommand{\Stab}{\mathrm{Stab}} 

\newcommand{\RR}{\mathbb{R}}
\newcommand{\QQ}{\mathbb{Q}}
\newcommand{\CC}{\mathbb{C}}
\newcommand{\ZZ}{\mathbb{Z}}

\newcommand{\Hom}{\mathrm{Hom}} 
\newcommand{\Supp}{\mathrm{Supp}} 
\newcommand{\ch}{\mathrm{ch}}  
\newcommand{\td}{\mathrm{td}} 
\newcommand{\Coh}{\mathrm{Coh}} 

\newcommand{\e}{\mathrm{e}}  
\newcommand{\NS}{\mathrm{NS}} 
\newcommand{\amp}{{\mathrm{Amp}}}   
\newcommand{\nef}{{\mathrm{Nef}}} 
\newcommand{\sHom}{\mathcal{H}om} 
 
\begin{document}

\title[Bayer--Macr\`i decomposition on Bridgeland moduli spaces]{Bayer--Macr\`i decomposition on\\ Bridgeland moduli spaces over surfaces}
\author[W.-M.~Liu]{Wanmin Liu
}
\address{Current Address: Center for Geometry and Physics, Institute for Basic Science (IBS), Pohang 37673, Republic of Korea}
\email{wanminliu@gmail.com}
\date{Kyoto J. Math. Volume  58, Number 3 (2018), 595-621.}

\begin{abstract}
We find a decomposition formula of the local Bayer--Macr\`i map for the nef line bundle theory on the Bridgeland moduli space over a surface. If there is a global Bayer--Macr\`i map, then such a decomposition gives a precise correspondence from Bridgeland walls to Mori walls. As an application, we compute the nef cone of the Hilbert scheme $S^{[n]}$ of $n$-points over special kinds of a fibered surface $S$ of Picard rank $2$.
\end{abstract}

\maketitle


\section{Introduction}\label{section-introduction}
Let $S$ be a smooth projective surface over $\CC$. Let $M$ be the Gieseker moduli space of semistable sheaves with fixed Chern character $\ch$ over $S$. One viewpoint for studying the birational geometry of $M$ is to study the classification of line bundles on $M$ and different cones inside its real N\'{e}ron--Severi group $N^1(M)$, such as the nef cone $\nef(M)$, and pseudoeffective cone $\overline{\mathrm{Eff}}(M)$. Another viewpoint is introduced by Bridgeland \cite{Bri07} with the idea of enlarging the category of coherent sheaves $\Coh(S)$ to its bounded derived version $\Db(S)$ and studying moduli spaces of Bridgeland semistable objects in $\Db(S)$.  Let $\sigma$ be a Bridgeland stability condition, and let $M_\sigma(\ch)$ be the moduli space of $\sigma$-semistable objects in $\Db(S)$ with the same invariant $\ch$. The collection of all stability conditions forms an interesting parameter space $\Stab(S)$, which is a $\CC$-manifold. Bridgeland identified $M$ as $M_\sigma(\ch)$, for $\sigma$ in a special chamber inside $\Stab(S)$. He envisioned that the wall-chamber structures of  $\Stab(S)$ will recover birational models of $M$.

When $S$ is the projective plane $\mathbb{P}^2$ and $\ch=(1,0,-n)$, the  moduli space $M$ is the Hilbert scheme $\mathbb{P}^{2[n]}$ of $n$-points over $\mathbb{P}^2$. Arcara, Bertram, Coskun, and Huizenga~\cite{ABCH13} found a precise relation between Bridgeland walls inside $\Stab(\mathbb{P}^2)$ with respect to $(1,0,-n)$ and Mori walls inside $\overline{\mathrm{Eff}}({\mathbb{P}^{2[n]}})$ (see Corollary~\ref{cor-ABCH}). Bertram and Coskun \cite{BC13} generalized the speculation to other rational surfaces. Bayer and Macr\`i \cite{BM14a, BM14b} linked the two viewpoints by establishing a line bundle theory on Bridgeland moduli spaces. Let $\sigma=\sigma_{\omega,\beta}$ be the stability condition constructed by Arcara and Bertram \cite{AB13}, which depends on an ample line bundle $\omega$ and another line bundle $\beta$ over $S$. Assume that $\sigma$ is in a chamber $\mathtt{C}$. Bayer and Macr\`i constructed a map by sending $\sigma$ to a nef line bundle $\ell_{\sigma}$ on $M_\sigma(\ch)$, which is called the local Bayer--Macr\`i map. The line bundle $\ell_{\sigma}$ only depends on the chamber $\mathtt{C}$. When $S$ is a K3 surface and $\ch$ is primitive, they constructed a global Bayer--Macr\`i map (by gluing the local Bayer--Macr\`i maps; see Definition~\ref{def-global-BM-map}),
$$\ell: \Stab^{\dag}(S)\to N^1(M),$$
sending a stability condition $\sigma$ to a line bundle $\ell_\sigma$ on $M$. The existence of the global Bayer--Macr\`i map is also known for the projective plane $\mathbb{P}^2$ with primitive Chern character $\ch$ by Li and Zhao \cite{LZ16}.

In this article, we find a decomposition of the line bundle $\ell_{\sigma_{\omega,\beta}}$. The decomposition is classified into two cases according to the given Chern character $\ch$. The case for objects supported in dimension $1$ is given in Lemma~\ref{lem-decomposition-dim1}. The case for objects supported in dimension $2$ is given in Lemma~\ref{lem-decomposition-dim2}. In this case, an equivalent decomposition is also obtained by Bolognese, Huizenga, Lin, Riedl, Schmidt, Woolf, and Zhao \cite[Proposition 3.8.]{BHL+16}. If there is a global Bayer--Macr\`i map, we then obtain the precise correspondences from Bridgeland walls to Mori walls for two such cases (see respectively, Theorems~\ref{thm-geometric-meaning-decomposition-dim1} and \ref{thm-geometric-meaning-decomposition-dim2}). By Mori walls, we mean the walls that appear on the stable base locus decomposition of the pseudoeffective cone $\overline{\mathrm{Eff}}(M)$. 

As an application of the main theorem, we compute the nef cone of the Hilbert scheme $S^{[n]}$ of $n$-points over special kinds of a fibered surface $S$ in Theorem~\ref{thm_nef-cone}. Here $S$ is either the Hirzebruch surface or an elliptic surface over $\mathbb{P}^1$ with a global section of Picard rank $2$. The example suggests that, to obtain the extremal nef line bundle, we cannot assume that the $\omega$ is parallel to $\beta$.
 
Some of the techniques discussed in this article have been partially generalized by Coskun and Huizenga \cite{CH15} to compute the nef cone of certain Gieseker moduli spaces.
 
\subsection*{Outline of the article} Section~\ref{section-Bridgeland-stability} is a brief review of the notion of Bridgeland stability conditions. Section~\ref{section-BM-theory} is a brief review of Bayer and Macr\`i's line bundle theory on Bridgeland moduli spaces. The main theorems on the Bayer--Macr\`i decomposition are given in Section~\ref{section-BM-decomposition}. In Section~\ref{section-toy-model}, we provide an application of the main theorem. Some background on the large-volume limit is given in Appendix~\ref{appendix-large-volume-limit}. Some parallel computations by using $\hat{Z}_{\omega,\beta}$ (see (\ref{eq-Z-twisted})) for a K3 surface are given in Appendix~\ref{appendix-K3}.

\subsection*{Acknowledgement}
This work is part of the author's Ph.D. dissertation at the Hong Kong University of Science and Technology. The author thanks his Ph.D. supervisor Wei-Ping Li for encouragement and guidance; Huai-Liang Chang, Naichung Conan Leung, Chunyi Li, Jason Lo, and Zhenbo Qin  and Xiaolei Zhao for helpful conversations and suggestions; Izzet Coskun for the invitation to the \emph{Graduate Student Workshop on Moduli Spaces and Bridgeland Stability} at the University of Illinois at Chicago in March 2013; Changzheng Li for the invitation to the Kavli Institute for the Physics and Mathematics of the Universe in October 2013; the organizers and participants of the workshop \emph{Geometry from Stability Conditions} at the University of Warwick in February 2015; and the anonymous referees for their valuable suggestions and comments. 

The author's work is supported by IBS-R003-D1.

\section{Bridgeland stability conditions}\label{section-Bridgeland-stability}
Let $S$ be a smooth projective surface over $\CC$, and let $\DS$ be the bounded derived category of coherent sheaves on $S$. Denote the Grothendieck group of $\DS$ by $K(S)$. 
A \emph{Bridgeland stability condition} (see \cite[Proposition 5.3]{Bri07}) $\sigma=(Z,\mathcal{A})$ on $\DS$ consists of a pair
$(Z,\mathcal{A})$, where $Z: K(S) \to \CC$ is a group homomorphism (called the \emph{central charge}) and $\mathcal{A}\subset\DS$ is the heart of a \emph{bounded $t$-structure} satisfying the following three properties.
\begin{enumerate}
\item[(1)] \emph{Positivity}. For any $0\neq E\in \mathcal{A}$ the central charge $Z(E)$ lies in the semiclosed upper half-plane $\RR_{>0}\cdot\e^{(0,1]\cdot i\pi}$. 
\end{enumerate}
Let $E\in \mathcal{A}\setminus\{0\}$. Define the \emph{Bridgeland slope} (might be $+\infty$-valued) and the \emph{phase} of $E$ as 
\begin{equation*}\label{eq-Bridgeland-slpoe}
\mu_{\sigma}(E):=\frac{-\Re(Z(E))}{\Im(Z(E))}; \quad \phi(E):=\frac{1}{\pi}\mathrm{arg}(Z(E))\in (0,1].
\end{equation*}
For nonzero $E, F\in \mathcal{A}$, we have the equivalent relation:
\begin{equation*}\label{BS}
\mu_{\sigma}(F)<(\leq)\mu_{\sigma}(E)\Longleftrightarrow \phi(F)<(\leq) \phi(E).
\end{equation*}
For $0\neq E\in \mathcal{A}$, we say $E$ is \emph{Bridgeland  (semi)stable} if for any subobject $0\neq F\subsetneq E$ ($0\neq F\subseteq E$) we have $\mu_{\sigma}(F)< (\leq)\mu_{\sigma}(E)$. \begin{enumerate}[resume]
\item[(2)] \emph{Harder--Narasimhan property}. Every object $E \in \mathcal{A}$ has a Harder--Narasimhan filtration $0 = E_0 \into E_1 \into \dots \into E_n = E$ such that the quotients $E_i/E_{i-1}$ are Bridgeland semistable with $\mu_{\sigma}(E_1/E_0) > \mu_{\sigma}(E_2/E_1) > \dots > \mu_{\sigma}(E_n/E_{n-1})$.
\item[(3)] \emph{Support property}. There is a constant $C>0$ such that, for all Bridgeland semistable objects $E\in \mathcal{A}$, we have
$\left\Vert E \right\Vert\leq C |Z(E)|$, where $\left\Vert \cdot \right\Vert$ is a fixed norm on $K(X)\otimes\RR$.
\end{enumerate}

\subsection{Bridgeland stability conditions on surfaces}
Let $S$ be a smooth projective surface. Fix $\omega,\beta\in N^1(S):=\NS(S)_{\RR}$ with $\omega$ ample. Define
\begin{equation*}\label{eq-Z}
Z_{\omega,\beta}(E):=-\int_{S} e^{-(\beta+\sqrt{-1}\omega)}.\ch(E).
\end{equation*}
For $E\in \Coh(S)$, denote its \emph{Mumford slope} by 
\begin{equation*}\label{eq-Mumford-slope}
\mu_{\omega}(E):=\left\{
\begin{array}{ll}
\frac{\omega\cdot\ch_1(E)}{\ch_0(E)} &\quad \text{if }\ch_0(E)\neq 0;\\
+\infty & \quad \text{otherwise.}
\end{array}
\right.
\end{equation*}
Let $\T_{\omega, \beta}\subset \Coh(S)$ be the subcategory of coherent sheaves whose Harder--Narasimhan (HN) factors (with respect to Mumford stability) are of Mumford slope strictly greater than $\omega.\beta$. Let $\F_{\omega, \beta}\subset \Coh(S)$ be the subcategory of coherent sheaves whose HN factors (with respect to Mumford stability) are of Mumford slope less than or equal to $\omega.\beta$. Then $(\T_{\omega, \beta},\F_{\omega, \beta})$ is a torsion pair of $\Coh(S)$ (see \cite{AB13}). Define the heart $\mathcal{A}_{\omega,\beta}$ as the tilt of this torsion pair:
\begin{equation*} 
\mathcal{A}_{\omega,\beta}:=\{E\in\Db(S):  H^{-1}(E)\in\F_{\omega,\beta}, H^0(E)\in\T_{\omega,\beta}, H^{p}(E)=0 \text{ otherwise}\}.
\end{equation*}
\begin{lemma}\cite[Corollary 2.1]{AB13} \label{BriStabCond} Fix $\omega,\beta\in \NS(S)_{\RR}$ with $\omega$ ample. Then
$\sigma_{\omega, \beta}:=(Z_{\omega, \beta}, \mathcal{A}_{\omega, \beta})$ is a Bridgeland stability condition.
\end{lemma}

\subsection{Logarithm Todd class} Let $X$ be a smooth projective variety over $\CC$. Let us introduce a formal variable $t$, and write 
\begin{eqnarray*}
\td(X)(t)&:=&1+(-\frac{1}{2}\KX) t+\frac{1}{12}\left(\frac{3}{2}\KX^2-\ch_2(X)\right)t^2\\
& &+\left(-\frac{1}{24}\KX.\left(\frac{1}{2}\KX^2-\ch_2(X)\right)\right)t^3+\text{ higher order of }t^4.
\end{eqnarray*}
Taking the logarithm with respect to $t$, and expressing it in the power series of $t$, we obtain
\begin{equation*}
 \ln \td(X)(t) =-\frac{1}{2}\KX t-\frac{1}{12}\ch_2(X) t^2+0\cdot t^3+\text{ higher order of }t^4.
\end{equation*}
In particular,  the \emph{logarithm Todd class} of a smooth projective surface $S$ or a smooth projective threefold $X$ is given, respectively, by
\begin{eqnarray*}
\ln\td(S) &:=& (0, -\frac{1}{2}\KS, -\frac{1}{12}\ch_2(S))\quad \text{ or}\\
\ln\td(X) &:=& (0, -\frac{1}{2}\KX, -\frac{1}{12}\ch_2(X), 0).
\end{eqnarray*}

\subsection{The Mukai pairing}\label{section-Mukai-bilinear-form}
We refer to \cite[Section 5.2]{Huy06} for the details. Let $X$ still be a smooth projective variety of dimension $n$ over $\CC$. Define the Mukai vector of an object $E\in \DX$ by
\begin{equation*}\label{Mukaivector}
v(E):= \ch(E).e^{\frac{1}{2}\ln\td(X)}\in \oplus H^{p,p}(X)\cap H^{2p}(X,\QQ)=:H^*_{\mathrm{alg}}(X,\QQ).
\end{equation*}
Let $A(X)$ be the Chow ring of $X$. The Chern character gives a mapping $\ch:K(X)\to A(X)\otimes \QQ$. There is a natural involution $^{\ast}: A(X) \to A(X)$,
\begin{equation*}\label{def-involution}
v=(v_0,\dots,v_i,\dots,v_n)\mapsto v^{\ast}:=(v_0,\dots,(-1)^i v_i,\dots,(-1)^n v_n).
\end{equation*}
We call $v^{\ast}$ the \emph{Mukai dual} of $v$. Denote $E^{\vee}:=R \sHom(E,\OO_S)$. We have
\begin{equation*}
\ch(E^\vee)=(\ch(E))^\ast,\quad v(E^{\vee})=(v(E))^{\ast}.e^{-\frac{1}{2}\KX}.
\end{equation*}
Define the \emph{Mukai pairing} for two Mukai vectors $w$ and $v$ by
\begin{equation}\label{Mukaipairing}
\langle w, v \rangle_X:=-\int_X w^{\ast}.v.e^{-\frac{1}{2}\KX}. 
\end{equation}
The Hirzebruch--Riemann--Roch theorem gives
\begin{equation*}\label{eq-HRR}
\chi(F,E)=\int_X \ch(F^{\vee}).\ch(E).\td(X)=-\langle v(F), v(E) \rangle_X.
\end{equation*}
For a smooth projective surface $S$, the Mukai vector of $E\in \DS$ is
\begin{eqnarray}\label{Mukaivectorsurface}
v(E)&=&(v_0(E), v_1(E),  v_2(E))\nonumber\\
&=&(\ch_0,\ch_1-\frac{1}{4}\ch_0\KS,\ch_2-\frac{1}{4}\ch_1.\KS+\frac{1}{2}\ch_0\left(\chi(\OO_S)-\frac{1}{16}\KS^2\right)). 
\end{eqnarray}
By (\ref{Mukaipairing}) the Mukai pairing of $w=(w_0, w_1, w_2)$ and $v=(v_0, v_1, v_2)$ is
\begin{equation}\label{Mukaipairingsurface}
\langle w, v \rangle_S 
= w_1.v_1-w_0 (v_2-\frac{1}{2}v_1.\KS)-v_0 (w_2+\frac{1}{2}w_1.\KS)-\frac{1}{8}w_0 v_0\KS^2. 
\end{equation}

\subsection{Central charge in terms of the Mukai pairing}\label{section-central-charge-Mukai}
\begin{lemma}\label{lem-central-charge} 
The central charge $Z_{\omega,\beta}$ has the expression:
\begin{equation}\label{Omega-Z}
Z_{\omega,\beta}(E)= \langle \mho_{Z_{\omega,\beta}}, v(E) \rangle_S, \text{ where }
\mho_{Z_{\omega,\beta}} := e^{\beta-\frac{3}{4}\KS+\sqrt{-1}\omega+\frac{1}{24}\ch_2(S)}.
\end{equation}
Moreover, the vector $\mho_{Z_{\omega,\beta}}$ (or simply $\mho_Z$) is given by
\begin{eqnarray}\label{Omega-Z-Mukai}
\mho_{Z_{\omega,\beta}}
&=& \left(1,\beta-\frac{3}{4}\KS, 
-\frac{1}{2}\omega^2+\frac{1}{2}(\beta-\frac{3}{4}\KS)^2-\frac{1}{2}\left(\chi(\OO_S)-\frac{1}{8}\KS^2\right)\right)  \nonumber\\
& & + \ \sqrt{-1}\left(0,\omega,(\beta-\frac{3}{4}\KS).\omega\right).
\end{eqnarray}
\end{lemma}
\begin{proof}
We have
\begin{eqnarray*}
Z_{\omega,\beta}(E)
&=& -\int_S e^{-(\beta+\sqrt{-1}\omega)}.\ch(E)\nonumber\\
&=& -\int_S e^{-\left(\beta+\ln\td(S)+\sqrt{-1}\omega\right)}.\sqrt{\td(S)}.\ch(E).\sqrt{\td(S)}. \nonumber\\
\end{eqnarray*}
Denote $\ch(F^\vee):=e^{-\left(\beta+\ln\td(S)+\sqrt{-1}\omega\right)}$. Then
\begin{eqnarray*}
\ch(F)^{*}&=& \ch(F^\vee)=e^{-\left(\beta-\frac{1}{2}\KS+\sqrt{-1}\omega\right)+\frac{1}{12}\ch_2(S)}\\
&=& \left(1,-(\beta-\frac{1}{2}\KS+\sqrt{-1}\omega),\frac{1}{2}(\beta-\frac{1}{2}\KS+\sqrt{-1}\omega)^2+\frac{1}{12}\ch_2(S)\right).
\end{eqnarray*}
So
\begin{eqnarray*}
\ch(F)&=& \left(1,(\beta-\frac{1}{2}\KS+\sqrt{-1}\omega),\frac{1}{2}(\beta-\frac{1}{2}\KS+\sqrt{-1}\omega)^2+\frac{1}{12}\ch_2(S)\right)\\
&=& e^{\left(\beta-\frac{1}{2}\KS+\sqrt{-1}\omega\right)+\frac{1}{12}\ch_2(S)}.
\end{eqnarray*}
Therefore,
\begin{equation*}
Z_{\omega,\beta}(E)
= -\int_S \ch(F^\vee).\sqrt{\td(S)}.\ch(E).\sqrt{\td(S)}=\langle v(F), v(E) \rangle_S.
\end{equation*}
So 
$$\mho_{Z_{\omega,\beta}}= v(F)=\ch(F).e^{\frac{1}{2}\ln\td(S)}=e^{\beta-\frac{3}{4}\KS+\sqrt{-1}\omega+\frac{1}{24}\ch_2(S)}.$$
By using Noether's formula
\begin{equation*}\label{eq-Noether}
-\frac{1}{12}\ch_2(S)=\chi(\OO_S)-\frac{1}{8}\KS^2
\end{equation*}
and direct computation, we get the concrete expression of $\mho_Z$.
\end{proof}

Denote by $\Stab(S)$ the collection of all Bridgeland stability conditions. It is a $\CC$-manifold of dimension $K_{\mathrm{num}}(S)\otimes\mathbb{C}$, with two group actions: a left action by $\mathrm{Aut}(\DS)$ and a right action by $\widetilde{\mathrm{GL}_2^+}(\RR)$ (see \cite[Lemma 8.2]{Bri07}). The stability $\sigma$ is said to be \emph{geometric} if all skyscraper sheaves $\OO_x$, $x\in S$, are $\sigma$-stable of the same phase. We can set the phase to be $1$ by a right group action. Denote by $\Stab^{\dag}(S)\subset \Stab(S)$ the connected component containing the geometric stability conditions. The stability $\sigma$ is said to be \emph{numerical} if the central charge $Z$ takes the form $Z(E)=\langle \pi(\sigma),v(E)\rangle_{S}$ for some vector $\pi(\sigma)\in K_{\mathrm{num}}(S)\otimes\mathbb{C}$. As in \cite[Remark 4.33]{Huy14}, we further \emph{assume} the numerical Bridgeland stability factors through $K_\mathrm{num}(S)_{\QQ}\otimes \CC \to H^*_{\mathrm{alg}}(S,\QQ)\otimes\CC$. Therefore $\pi(\sigma)\in H^*_{\mathrm{alg}}(S,\QQ)\otimes\CC$. For a numerical geometric stability condition with skyscraper sheaves of phase $1$, the heart $\A$ must be of the form $\A_{\omega,\beta}$ (see \cite[Proposition 10.3]{Bri07} and Huybrechts \cite[Theorem 4.39]{Huy14}). Therefore, Lemma~\ref{lem-central-charge} gives
\begin{equation}
\label{eq-forgetful-map}
\pi(\sigma_{\omega,\beta})=\mho_{Z_{\omega,\beta}}\in H^*_{\mathrm{alg}}(S,\QQ)\otimes\CC.
\end{equation}

\subsection{Bertram's nested wall theorem} \label{section-Maciocia} 
We follow the notation from \cite[Section 2]{Mac14} (but use $H$ instead of $\omega$ therein). Fix an ample divisor $H$ and another divisor $\gamma\in H^{\perp}$, that is,  $H.\gamma=0$.
Denote
\begin{equation*}\label{eq-g-d}
g:=H^2, \quad -d:=\gamma^2.
\end{equation*} 
It is known by the \emph{Hodge index theorem} that $d\geq 0$ and that $d=0$ if and only if $\gamma=0$. Let $\ch=(\ch_0,\ch_1,\ch_2)$ be of Bogomolov type, that is, $\ch_1^2-2\ch_0\ch_2\geq 0$. Write it as
\begin{equation*}\label{eq-ch-Maciocia}
\ch=(\ch_0,\ch_1,\ch_2):=(x,y_1 H+ y_2 \gamma + \delta,z),
\end{equation*}
where $y_1$, $y_2$ are real coefficients and $\delta \in \{ H,\gamma \}^{\perp}$. Write the potential destabilizing Chern character as
\begin{equation*}\label{eq-ch'-Maciocia}
\ch'=(\ch_0',\ch_1',\ch_2'):=(r,c_1 H+ c_2 \gamma+\delta',\chi),
\end{equation*}
where $c_1$, $c_2$ are real coefficients and $\delta' \in \{ H,\gamma \}^{\perp}$. A \emph{potential wall} is defined as
$$W(\ch,\ch'):=\{\sigma\in \Stab(S)|\, \mu_\sigma(\ch)=\mu_\sigma(\ch')\}.$$
A potential wall $W(\ch,\ch')$ is a \emph{Bridgeland wall} if there are a $\sigma\in\Stab(S)$ and objects $E, F\in\A_\sigma$ such that $\ch(E)=\ch$, $\ch(F)=\ch'$ and $\mu_\sigma(E)=\mu_\sigma(F)$. There is a wall-chamber structure on $\Stab(S)$ with respect to $\ch$ (see \cite{Bri07,Bri08,Tod08}). Bridgeland walls are $\mathbb{R}$-codimension $1$ in $\Stab(S)$ and separate $\Stab(S)$ into chambers. Let $E$ be an object that is $\sigma_0$-stable for a stability condition $\sigma_0$ in some chamber $\mathtt{C}$. Then $E$ is $\sigma$-stable for any $\sigma\in\mathtt{C}$. 
Choose
\begin{equation}\label{eq-coordinate}
\left\{
\begin{array}{ll}
\omega:=tH, \\
\beta:=sH+u\gamma,
\end{array}
\right.
\end{equation}
for some real numbers $t$, $s$, $u$, with $t$ positive. With a sign choice of $\gamma$, we further assume $u\geq 0$. There is a half-three-space of stability conditions
\begin{equation*}\label{eq-Maciocia-3space}
\Omega_{\omega,\beta}=\Omega_{tH,sH+u\gamma}:=
\{ \sigma_{tH,sH+u\gamma}\ |\ t>0, u\geq 0 \} \subset \Stab^{\dag}(S),
\end{equation*}
which should be considered to be the $u$-indexed family of 
half-planes
$$\Pi_{(H,\gamma,u)}:=\{ \sigma_{tH,sH+u\gamma}\ |\ t>0,\ u \text{ is fixed} \}.$$
\begin{definition} \label{def-coordinate}
A \emph{frame} with respect to the triple $(H,\gamma,u)$ is a choice of an ample divisor $H$ on $S$, another divisor $\gamma\in H^\perp$, and a nonnegative number $u$ such that the stability conditions $\sigma_{\omega,\beta}$ are on the half-plane $\Pi_{(H,\gamma,u)}$ with $(s,t)$-coordinates as in (\ref{eq-coordinate}). We simply call this fixing a frame $(H,\gamma,u)$, and write $\sigma_{s,t}:=\sigma_{tH,sH+u\gamma}$.
\end{definition}

\begin{theorem}[Bertram's nested wall theorem in $(s,t)$-model]\label{thm-Maciocia} 
\cite[Section 2]{Mac14}
Fix a frame $(H,\gamma,u)$. The potential walls $W(\ch,\ch')$ (for the fixed $\ch$ and different potential destabilizing Chern characters $\ch'$) in the $(s,t)$-half-plane $\Pi_{(H,\gamma,u)}$ ($t>0$) are given by nested semicircles with center $(C,0)$ and radius $R=\sqrt{D+C^2}$:
\begin{equation}\label{eq-wall-equation}
(s-C)^2+t^2=D+C^2,
\end{equation}
where $C=C(\ch,\ch')$ and $D=D(\ch,\ch')$ are given by
\begin{equation}\label{eq-C_sigma}
C(\ch,\ch'):=\frac{x\chi-rz +ud(xc_2-ry_2)}{g(xc_1-ry_1)},
\end{equation}
\begin{equation}\label{eq-D_sigma}
D(\ch,\ch'):=\frac {2zc_1-2c_{2}udy_{1}-xu^{2}dc_1+2y_{2}udc_{1}-2\chi y_{1}+ru^{2}dy_{1}}{g(xc_1-ry_1)}.
\end{equation}
\begin{itemize}
\item If $\ch_0=x\neq 0$, we have
\begin{eqnarray}
D&=&-\frac{2y_1}{x}C+\frac{ud(2y_2-ux)+2z}{gx} \label{eq-D-C}\\
&=&-\frac{2y_1}{x}C+(\frac{y_1^2}{x^2}-F), \label{eq-D-C-F} 
\end{eqnarray}
where $F=F(\ch)$ is independent of $\ch'$,
\begin{equation}\label{eq-F}
F(\ch):=\frac{d}{g}\left(u-\frac{y_2}{x}\right)^2+\frac{1}{x^2 g}(y_1^2 g - y_2^2 d - 2xz).
\end{equation}
Moreover, if $\ch$ is of Bogomolov type, that is, $\ch_1^2-2\ch_0\ch_2\geq 0$, then $F(\ch)\geq 0$ for all $u$.

\item If $\ch_0=0$ and $\ch_1.H>0$, that is, $x=0$ and $y_1>0$, then $\ch_0'=r\neq 0$, and $C=\frac{z+duy_2}{gy_1}$ is independent of $\ch'$. We have
\begin{eqnarray}
D&=&-\frac{2c_1}{r}C+\frac{ud(2c_2-ur)+2\chi}{gr}  \label{eq-D-C-r} \\
&=&-\frac{2c_1}{r}C+(\frac{c_1^2}{r^2}-F'), \label{eq-D-C-r-F} 
\end{eqnarray}
where $F'=F'(\ch')$ is independent of $\ch$,
\begin{equation}\label{eq-F'}
F'(\ch'):=\frac{d}{g}\left(u-\frac{c_2}{r}\right)^2+\frac{1}{r^2 g}(c_1^2 g - c_2^2 d - 2r\chi).
\end{equation}
Moreover, if $\ch'$ is of Bogomolov type, then $F(\ch')\geq 0$ for all $u$.
\end{itemize}
\end{theorem}
\begin{proof}
We refer to Maciocia \cite[Section 2]{Mac14}. The only unproved parts are (\ref{eq-D-C-r}) and (\ref{eq-D-C-r-F}). It is an easy exercise to check them.
\end{proof}

\subsection{From $(s, t)$-model to $(s, q)$-model} We follow the ideas of Li--Zhao \cite{LZ16} and consider a $\widetilde{\mathrm{GL}_2^+}(\RR)$-action on $\sigma_{\omega, \beta}$. The potential walls in the $(s,q)$-plane are semilines.
\begin{definition}
 Fix a frame $(H,\gamma,u)$. Define $\sigma'_{\omega,\beta}=(Z'_{\omega,\beta}, \A'_{\omega, \beta})$ as the right action of $\begin{pmatrix}
1 & 0\\
-\frac{s}{t} & \frac{1}{t}
\end{pmatrix}$ on $\sigma_{\omega,\beta}$, that is, $\A'_{\omega, \beta}=\A_{\omega, \beta}$ and
\begin{equation} \label{eq-Z-Z'}
Z{'}_{\omega,\beta}(E):=\left(\Re Z_{\omega,\beta}(E)-\frac{s}{t}\Im Z_{\omega,\beta}(E)\right)+\frac{1}{t}i\Im Z_{\omega,\beta}(E).
\end{equation}
\end{definition}
\begin{lemma}
Fix a frame $(H,\gamma,u)$. The above right action does not change the potential walls $W(\ch,\ch')$ in the $(s,t)$-plane $\Pi_{(H,\gamma,u)}$.
\end{lemma}
\begin{proof}
This is a direct computation because the potential wall relation for $Z'_{\omega,\beta}$ is equivalent to the potential wall relation for $Z_{\omega,\beta}$ by using (\ref{eq-Z-Z'}):
\begin{eqnarray*}
& &\Re Z' (\ch') \Im Z' (\ch)-\Re Z' (\ch)\Im Z' (\ch')=0 \nonumber \\
&\Leftrightarrow & \Re Z(\ch') \Im Z (\ch)-\Re Z (\ch)\Im Z (\ch')=0.
\end{eqnarray*}
\end{proof}

\begin{definition} Fix a frame $(H,\gamma,u)$. We change the $(s, t)$-plane $\Pi_{(H,\gamma,u)}$ to the $(s, q)$-plane $\Sigma_{(H,\gamma,u)}$ by keeping the same $s$ and defining
\begin{equation} \label{eq-q}
q:=\frac{s^2+t^2}{2}.
\end{equation}
Denote $\sigma_{s,q}:=\sigma'_{tH,sH+u\gamma}$.
The central charge (\ref{eq-Z-Z'}) becomes
\begin{eqnarray*}\label{eq-central-charge}
Z_{s,q}(E)&=&(-\ch_2(E)+\ch_0(E) H^2 q)+\left(-\frac{1}{2}\ch_0(E) \gamma^2 u^2 + u \ch_1(E).\gamma\right)\nonumber\\
& &+i(\ch_1(E).H-\ch_0(E) H^2 s).
\end{eqnarray*}
\end{definition}

\begin{corollary}[Bertram's nested wall theorem in $(s,q)$-model] Fix a frame $(H,\gamma,u)$, and use the notation as above. The potential walls $W(\ch,\ch')$ in the $(s,q)$-plane $\Sigma_{(H,\gamma,u)}$ are given by semilines
\begin{equation*}\label{eq-wall-s-q}
q=Cs+\frac{1}{2}D \quad (q>\frac{s^2}{2}).
\end{equation*}
\begin{itemize}
\item If $x\neq 0$, then the potential walls are given by semilines passing through a fixed point $(\frac{y_1}{x}, \frac{1}{2}\left(\frac{y_1^2}{x^2}-F\right))$ with slope $C=C(\ch,\ch')$:
\begin{equation}\label{eq-wall-passing-fixed-point}
q=C(s-\frac{y_1}{x})+\frac{1}{2}\left(\frac{y_1^2}{x^2}-F\right), \quad (q>\frac{s^2}{2}),
\end{equation}
where $F=F(\ch)$ as in (\ref{eq-F}) is independent of $\ch'$.
\item If $x=0$ and $y_1>0$, then $r\neq 0$. The potential walls are given by parallel semilines with constant slope $C=\frac{z+duy_2}{gy_1}$:
\begin{equation}\label{eq-wall-parallel-lines}
q= C (s-\frac{c_1}{r})+\frac{1}{2}\left(\frac{c_1^2}{r^2}-F'\right), \quad (q>\frac{s^2}{2}),
\end{equation}
where $F'=F'(\ch')$ as in (\ref{eq-F'}) is independent of $\ch$.
\end{itemize} 
\end{corollary}
\begin{proof}
This is a direct computation by using (\ref{eq-wall-equation}) and (\ref{eq-q}). 
\end{proof}
\begin{remark}In the case of $\mathbb{P}^2$, the condition $q>\frac{s^2}{2}$ is relaxed, $q$ could be a little negative, and the boundary is given by a fractal curve (see \cite{LZ16}).
\end{remark}

\subsection{Duality induced by derived dual}
\begin{lemma}\cite[Theorem 3.1]{Mar17} \label{lem-duality}
The functor $\Phi(\cdot):=R\sHom(\cdot,\OO_S)[1]$ induces an isomorphism between the Bridgeland moduli spaces $M_{\omega,\beta}(\ch)$ and $M_{\omega,-\beta}(-\ch^*)$ provided these moduli spaces exist and $Z_{\omega,\beta}(\ch)$ belongs to the open upper half-plane.
\end{lemma}
\begin{proof}
This is a variation of Martinez's duality theorem \cite[Theorem 3.1]{Mar17}, where the duality functor is taken as $R\sHom(\cdot,\omega_S)[1]$.
\end{proof}

\begin{corollary}\label{cor-derived-dual-stability}
Fix the Chern character $\ch=(\ch_0,\ch_1,\ch_2)$. Assume that $Z_{\omega,\beta}(\ch)$ belongs to the open upper half-plane.  The wall-chamber structures of $\sigma_{\omega,\beta}$ with respect to $\ch$ are dual to the wall-chamber structures of $\Phi(\sigma_{\omega,\beta})$ with respect to $\Phi(\ch)=-\ch^*=(-\ch_0,\ch_1,-\ch_2)$ in the sense that
\begin{equation*}
\Phi(\sigma_{\omega,\beta})=\sigma_{\omega,-\beta}.
\end{equation*}
Applying $\Phi$ again, we have $\Phi\circ\Phi(\sigma_{\omega,\beta})=\sigma_{\omega,\beta}$. Moreover, if we fix a frame $(H,\gamma,u)$, then $\sigma_{\omega,\beta} \in \Pi_{(H,\gamma,u)}$ with coordinates $(s,t)$ is dual to $\Phi(\sigma_{\omega,\beta}) \in \Pi_{(H,-\gamma,u)}$ with coordinates $(-s,t)$.
\begin{itemize}
\item If $\sigma_{\omega,\beta}\in \mathtt{C}$, where $\mathtt{C}$ is a chamber with respect to $\ch$ in $\Pi_{(H,\gamma,u)}$, then we have $\Phi(\sigma_{\omega,\beta})\in \mathtt{DC}$, where $\mathtt{DC}$ is the corresponding chamber with respect to $\Phi(\ch)$ in $\Pi_{(H,-\gamma,u)}$. 
\item If $\sigma:=\sigma_{\omega,\beta}\in W(\ch,\ch')$ in $\Pi_{(H,\gamma,u)}$, then $\Phi(\sigma)\in W(-\ch^*,-{\ch'}^*)$ in $\Pi_{(H,-\gamma,u)}$, and there are relations\\ 
$\mu_{\Phi(\sigma)}(-\ch^*)=-\mu_{\sigma}(\ch)$, $C_{\Phi(\sigma)}(-\ch^*,-{\ch'}^*)=-C_{\sigma}(\ch,\ch')$,\\
$D_{\Phi(\sigma)}(-\ch^*,-{\ch'}^*)=D_{\sigma}(\ch,\ch')$, \\
and $R_{\Phi(\sigma)}(-\ch^*,-{\ch'}^*)=R_{\sigma}(\ch,\ch')$.
\end{itemize}
\end{corollary}
\begin{proof}
The proof is a direct computation.
\end{proof}
\begin{remark}\label{rmk-upper-half-plane}
The assumption that $Z_{\omega,\beta}(\ch)$ belongs to the open upper half-plane means exactly that we exclude the case $\Im Z_{\omega,\beta}(\ch)=0$, which is equivalent to one of the following three subcases:
\begin{itemize}
\item $\ch=(0,0,n)$ for some positive integer $n$;
\item $\ch_0>0$ and $\Im Z_{\omega,\beta}(\ch)=0$; or
\item $\ch_0<0$ and $\Im Z_{\omega,\beta}(\ch)=0$.
\end{itemize}
We call the first subcase the trivial chamber, the second subcase the Uhlenbeck wall, and the third subcase the dual Uhlenbeck wall (see Definition~\ref{def-Gieseker-chamber}). 
\end{remark}

\section{Bayer--Macr\`i's nef line bundle theory}\label{section-BM-theory}
\subsection{The local Bayer-Macr{\`i} map}
Let $S$ be a smooth projective surface over $\mathbb{C}$. Let $\sigma=(Z,\A)\in \Stab(S)$ be a stability condition, and let $\ch=(\ch_0,\ch_1,\ch_2)$ be a choice of Chern character. Assume that we are given a \emph{flat family} (see \cite[Definition 3.1]{BM14a}) $\mathcal{E}\in \Db(M\times S)$ of $\sigma$-semistable objects of class $\ch$ parameterized by a proper algebraic space $M$ of finite type over $\mathbb{C}$. Denote by $N^1(M)=\NS(M)_\RR$ the group of real Cartier divisors modulo numerical equivalence. Write $N_1(M)$ as the group of real $1$-cycles modulo numerical equivalence with respect to the intersection pairing with Cartier divisors. The Bayer--Macr\`i's numerical Cartier divisor class $\ell_{\sigma, \mathcal{E}}\in N^1(M) = \Hom(N_1(M), \mathbb{R})$ is defined as follows: for any projective integral curve $C\subset M$, 
\begin{equation} \label{Bayer-Macri-def}
\ell_{\sigma,\mathcal{E}}([C]) =  \ell_{\sigma,\mathcal{E}}.C := \Im \left(-\frac{Z\bigl(\Phi_{\mathcal{E}}(\OO_C)\bigr)}{Z(\ch)}\right)
= \Im \left(-\frac{Z\bigl((p_S)_* \mathcal{E}|_{C \times S}\bigr)}{Z(\ch)} \right),
\end{equation}
where $\Phi_\mathcal{E}\colon \Db(M)\to\Db(S)$ is the Fourier--Mukai functor with kernel $\mathcal{E}$ and $\OO_C$ is the structure sheaf of $C$.

\begin{theorem}\label{thm-Bayer-Macri} \cite[Theorem 1.1]{BM14a}
The divisor class $\ell_{\sigma, \mathcal{E}}$ is nef on $M$.
In addition, we have $\ell_{\sigma, \mathcal{E}}.C = 0$ if and only if for two general points $c, c' \in C$, the corresponding objects $\mathcal{E}_c, \mathcal{E}_{c'}$ are $S$-equivalent.
\end{theorem}
Here two semistable objects are \emph{$S$-equivalent} if their Jordan--H\"older filtrations into stable factors of the same phase have identical stable factors.

\begin{definition} Let $\mathtt{C}$ be a Bridgeland chamber with respect to $\ch$. Assume the existence of the moduli space $M_{\sigma}(\ch)$ for $\sigma\in \mathtt{C}$ with a universal family $\mathcal{E}$. Then $M_{\mathtt{C}}(\ch):=M_\sigma(\ch)$ is constant for $\sigma\in \mathtt{C}$. Theorem~\ref{thm-Bayer-Macri} yields a map
\begin{eqnarray*}
\ell : \overline{\mathtt{C}} & \longrightarrow & \nef(M_{\mathtt{C}}(\ch))\\
\sigma & \mapsto & \ell_{\sigma,\mathcal{E}}
\end{eqnarray*}
which is called the \emph{local Bayer--Macr\`i map} for the chamber $\mathtt{C}$ with respect to $\ch$.
\end{definition} 

For any $\sigma\in \Stab^{\dag}(S)$, after a $\widetilde{\mathrm{GL}_2^+}(\RR)$-action, we assume that $\sigma=\sigma_{\omega,\beta}$, that is, skyscraper sheaves are stable of phase $1$. Denote 
$$\mathbf{v}:=v(\ch)=\ch\cdot e^{\frac{1}{2}\ln\td(S)}.$$
The local Bayer--Macr\`i map is the composition of the following three maps:
\begin{equation*}
\Stab^{\dag}(S) \xrightarrow{\pi} H^*_{\mathrm{alg}}(S,\QQ)\otimes\CC \xrightarrow{\mathcal{I}} \mathbf{v}^{\perp}\xrightarrow{\theta_{\mathtt{C},\mathcal{E}}} N^1 (M_\mathtt{C}(\ch)).
\end{equation*}

\begin{itemize}
\item The map $\pi$ forgets the heart: $\pi(\sigma_{\omega,\beta}):=\mho_{Z_{\omega,\beta}}$ as in (\ref{eq-forgetful-map}).
\item For any $\mho\in H^*_{\mathrm{alg}}(S,\QQ)\otimes\CC$, define $\mathcal{I}(\mho):=\Im \frac{\mho}{-\langle \mho, \mathbf{v} \rangle_S}$. One can check that $\mathcal{I}(\mho)\in \mathbf{v}^{\perp}$ (this also follows from Lemma~\ref{lem-w-sigma-1}), where the perpendicular relation is with respect to the Mukai pairing: 
\begin{equation}\label{eq-v-perpendicular}
\mathbf{v}^{\perp}:=\{w\in H^*_{\mathrm{alg}}(S,\QQ)\otimes\RR \ | \ \langle w, \mathbf{v}\rangle_S =0\}.
\end{equation}
\item The third map $\theta_{\mathtt{C},\mathcal{E}}$ is the \emph{algebraic Mukai morphism}. More precisely, for a fixed Mukai vector $w\in \mathbf{v}^{\perp}$ and an integral curve $C\subset M_\mathtt{C}(\ch)$, 
\begin{equation*}\label{def-Mukai-morphism}
\theta_{\mathtt{C},\mathcal{E}}(w).[C]:=\langle w, v(\Phi_{\mathcal{E}}(\OO_C))\rangle_S.
\end{equation*}
\end{itemize}
\begin{definition}
Define $w_{\sigma_{\omega,\beta}}(\ch):=-\Im \left(\overline{\langle \mho_Z, \mathbf{v} \rangle_S}\cdot \mho_Z\right)$. We simply write it as $w_{\omega,\beta}$ or $w_\sigma$.  
\end{definition}
\begin{lemma} \label{lem-w-sigma-1}
Fix the Chern character $\ch=(\ch_0,\ch_1,\ch_2)$. The line bundle class $\ell_{\sigma_{\omega,\beta}}\in N^1(M_{\sigma_{\omega,\beta}}(\ch))$ (if it exists) is given by
\begin{equation}\label{eq-linebundle-mukaivector}
\ell_{\sigma_{\omega,\beta}}\xlongequal{\RR_+}\theta_{\sigma,\mathcal{E}}(w_{\omega,\beta}),
\end{equation}
where $w_{\omega,\beta}\in \mathbf{v}^{\perp}$ is given by
\begin{equation}\label{eq-w}
w_{\omega,\beta}=\left(\Im Z(\ch)\right)\Re \mho_Z-\left(\Re Z(\ch)\right)\Im \mho_Z.
\end{equation}
\end{lemma}
\begin{proof} By the definition, $w_{\omega,\beta} = |Z(E)|^2 \mathcal{I}(\mho_Z)$. Applying the Mukai morphism, we get (\ref{eq-linebundle-mukaivector}). Taking the complex conjugate of (\ref{Omega-Z}) we get $\overline{\langle \mho_Z, \mathbf{v} \rangle_S}=\Re Z(\ch)-\sqrt{-1}\Im Z(\ch)$. The relation (\ref{eq-w}) thus follows from the definition of $w_{\omega,\beta}$. 
By the definition of $\mho_Z$, we have $\langle\Re \mho_Z,\mathbf{v}\rangle_S+\sqrt{-1}\langle\Im \mho_Z,\mathbf{v}\rangle_S=\langle \mho_Z,\mathbf{v}\rangle_S=\Re Z(\ch)+\sqrt{-1}\Im Z(\ch)$. We then obtain the perpendicular relation
$\langle w_\sigma, \mathbf{v}\rangle_S=\left(\Im Z(\ch)\right)\langle\Re \mho_Z,\mathbf{v}\rangle_S-\left(\Re Z(\ch)\right)\langle\Im \mho_Z,\mathbf{v}\rangle_S=0$.
\end{proof}

The Mukai morphism is the dual version of the Donaldson morphism (see \cite[Proposition 4.4, Remark 5.5]{BM14a}). The surjectivity of the Mukai morphism is not known in general. We will compute the image of the local Bayer--Macr\`i map in Theorem~\ref{thm-geometric-meaning-decomposition-dim2}. 

Let $\mathcal{E}$ be a universal family over $M_{\sigma}(\ch)$. Denote by $\mathcal{F}$ the dual universal family over $M_{\Phi(\sigma)}(-\ch^*)$. Then $w_{\omega,-\beta}(-\ch^*)\in v(-\ch^*)^\perp$. 
\begin{lemma} \label{lem-duality-line-bundle}
Fix the Chern character $\ch$. Let $\sigma:=\sigma_{\omega,\beta}$, and assume that $Z_{\omega,\beta}(\ch)$ belongs to the open upper half-plane (as Remark~\ref{rmk-upper-half-plane}). Then 
\begin{equation*}
\ell_{\sigma}\cong \ell_{\Phi(\sigma)}, \text{ that is, } \theta_{\sigma,\mathcal{E}}(w_{\omega,\beta}(\ch))\cong \theta_{\Phi(\sigma),\mathcal{F}}(w_{\omega,-\beta}(-\ch^*)).
\end{equation*}
\end{lemma}
\begin{proof}
This is a consequence of the isomorphism of the moduli spaces
$$M_{\sigma_{\omega,\beta}}(\ch)\cong M_{\sigma_{\omega,-\beta}}(-\ch^*)$$
induced by the duality functor $\Phi(\cdot)=R\sHom(\cdot,\OO_S)[1]$.
\end{proof}

\subsection{The global Bayer--Macr\`i map} \label{section-identification-NS-group}
Let $\sigma$ be in a chamber $\mathtt{C}$. The line bundle $\ell_{\sigma,\mathcal{E}}$ is only defined \emph{locally}, that is,  $\ell_{\sigma,\mathcal{E}}\in N^1(M_\mathtt{C}(\ch))$. If we take another chamber $\mathtt{C}'$, we cannot say $\ell_{\sigma,\mathcal{E}}\in N^1(M_{\mathtt{C}'}(\ch))$ directly. We want to associate to $\ell_{\sigma,\mathcal{E}}$ the \emph{global} meaning in the following way. 

Let $\sigma\in \mathtt{C}$ and $\tau\in\mathtt{C}'$ be two generic numerical stability conditions in different chambers with respect to $\ch$. \emph{Assume} $M_\sigma(\ch)$ and $M_\tau(\ch)$ are nonempty and irreducible with universal families $\mathcal{E}$ and $\mathcal{F}$, respectively. And \emph{assume} that there  is a birational map between $M_\sigma(\ch)$ and $M_\tau(\ch)$, induced by a derived autoequivalence $\Psi$ of $\DS$ in the following sense: there exists a common open subset $U$ of $M_\sigma(\ch)$ and $M_\tau(\ch)$, with complements of \emph{codimension at least} $2$, such that, for any $u\in U$, the corresponding objects $\mathcal{E}_u\in M_\sigma(\ch)$ and $\mathcal{F}_u\in M_\tau(\ch)$ are related by $\mathcal{F}_u=\Psi(\mathcal{E}_u)$. Then the N\'{e}ron--Severi groups of $M_\sigma(\ch)$ and $M_\tau(\ch)$ can canonically be identified. So for a Mukai vector $w\in\mathbf{v}^{\perp}$, the two line bundles $\theta_{\mathtt{C},\mathcal{E}}(w)$ and $\theta_{\mathtt{C}',\mathcal{F}}(w)$ are identified. 

\begin{definition}\label{def-global-BM-map}
Fix a base geometric numerical stability condition $\sigma$ in a chamber. A \emph{global} Bayer--Macr\`i map $$\ell: \Stab^{\dag}(S)\to N^1(M_\sigma(\ch)),$$
is glued by the \emph{local} Bayer--Macr\`i map by the above identification.
\end{definition}

\begin{theorem} (\cite[Theorem 1.2]{BM14b} for K3 surface, \cite[Theorem 0.2]{LZ16} for $\mathbb{P}^2$)\label{theorem-global-BM}
Let $S$ be a K3 surface or the projective plane $\mathbb{P}^2$. Let $\ch$ be primitive character over $S$. There is a global Bayer--Macr\`i map. 
\end{theorem}

\section{Bayer--Macr\`i decomposition}\label{section-BM-decomposition}
In this section, we give an \emph{intrinsic} decomposition of the Mukai vector $w_{\omega,\beta}(\ch)$ in Lemmas \ref{lem-decomposition-dim1} and \ref{lem-decomposition-dim2}, respectively, according to the dimension of the support of objects with invariants $\ch$. In particular, each component is in $\mathbf{v}^\perp$. So we can apply $\theta_{\sigma,\mathcal{E}}$ and obtain the \emph{intrinsic} decomposition of $\ell_{\sigma_{\omega,\beta}}$. We call such a decomposition of $w_{\omega,\beta}$ or $\ell_{\sigma_{\omega,\beta}}$ the \emph{Bayer--Macr\`i decomposition}.

\subsection{Preliminary computation by using $\mho_Z$}
\begin{lemma} \label{lem-w-sigma}
If $\Im Z(\ch)=0$, then $w_\sigma \xlongequal{\RR_+} \Im \mho_Z$. If $\Im Z(\ch)>0$, then 
\begin{eqnarray}\label{eq-w_sigma-decomposition}
w_\sigma &\xlongequal{\RR_+}& \mu_\sigma(\ch)\Im \mho_Z+\Re \mho_Z \\
         &=& \left(0,\mu_\sigma(\ch)\omega+\beta,-\frac{3}{4}\KS.(\mu_\sigma(\ch)\omega+\beta)\right)\nonumber\\
         & & +\left(1,-\frac{3}{4}\KS,-\frac{1}{2}\chi(\OO_S)+\frac{11}{32}\KS^2)\right)\nonumber\\
         & & +\left(0,0,\beta.(\mu_\sigma(\ch)\omega+\beta)-\frac{1}{2}(\omega^2+\beta^2)\right).\nonumber
\end{eqnarray} 
\end{lemma}
\begin{proof} The case for $\Im Z(\ch)=0$ follows from (\ref{eq-w}). If $\Im Z(\ch)>0$, we divide (\ref{eq-w}) by this positive number and obtain (\ref{eq-w_sigma-decomposition}). The concrete formula is then derived by (\ref{Omega-Z-Mukai}).
\end{proof}

\begin{lemma}\label{lem-direct-relations-using-centers}
Fix a frame $(H,\gamma,u)$. We have relations
\begin{eqnarray}
\mu_{\sigma}(\ch)\omega+\beta &=& C(\ch,\ch')H+u\gamma, \label{eq-relation1}
\\
\beta.(\mu_\sigma(\ch)\omega+\beta)-\frac{1}{2}(\omega^2+\beta^2) &=& -\frac{g}{2}D(\ch,\ch')-\frac{d}{2}u^2, \label{eq-relation2}
\end{eqnarray}
where the numbers $C(\ch,\ch')$ and  $D(\ch,\ch')$ are given by (\ref{eq-C_sigma}) and (\ref{eq-D_sigma}).
\end{lemma}
\begin{proof}
The proof is a direct computation by using Maciocia's Theorem~\ref{thm-Maciocia}. For the reader's convenience, we give the details. 
For (\ref{eq-relation1}), we only need to check that
\begin{equation*}
\mu_{\sigma}(\ch)t+s = C(\ch,\ch').
\end{equation*}
Recall that the wall equation is $(s-C)^2+t^2=D+C^2$. Now
\begin{eqnarray*}
\mu_{\sigma}(\ch)t+s&=&\frac{z-sy_1 g+uy_2d+\frac{x}{2}\left(s^2g-u^2d-t^2g\right)}{(y_1-xs)g}+s \\
&=&\frac{z+uy_2d-\frac{x}{2}u^2d-\frac{xg}{2}(s^2+t^2)}{(y_1-xs)g}\\
&=&\frac{z+uy_2d-\frac{x}{2}u^2d-\frac{xg}{2}(2sC+D)} {(y_1-xs)g}\quad \text{ by using the wall equation.}
\end{eqnarray*}
So we only need to check that
\begin{equation}\label{eq-check}
z+uy_2d-\frac{x}{2}u^2d-\frac{xg}{2}(2sC+D)=(y_1-xs)gC.
\end{equation}
If $x=0$, then (\ref{eq-check}) is true since $C=\frac{z+duy_2}{gy_1}$. If $x\neq 0$, then (\ref{eq-check}) is still true by using (\ref{eq-D-C}).

Let us prove (\ref{eq-relation2}). We have that
\begin{eqnarray*}
\textrm{LHS of (\ref{eq-relation2})}&=&
(sH+u\gamma).(CH+u\gamma)-\frac{g}{2}(t^2+s^2-\frac{d}{g}u^2)\\
&=& sCg-u^2 d-\frac{g}{2}(2sC+D-\frac{d}{g}u^2)=\textrm{RHS of (\ref{eq-relation2})}
\end{eqnarray*}
\end{proof}

\begin{definition}
Fix a frame $(H,\gamma,u)$.  Define the vector $\mathbf{t}_{(H,\gamma,u)}(\ch, \ch')$ as
\begin{equation*}
\left(1,CH+u\gamma-\frac{3}{4}\KS, -\frac{3}{4}\KS.(CH+u\gamma)-\frac{1}{2}\chi(\OO_S)+\frac{11}{32}\KS^2\right),
\end{equation*}
where the center $C=C(\ch,\ch')$ is as in (\ref{eq-C_sigma}).
\end{definition}

\begin{lemma} If $\Im Z_{\omega,\beta}(\ch)>0$, then
\begin{equation}\label{eq-w-sigma-general}
w_{\sigma\in W(\ch,\ch')} \xlongequal{\RR_+} \left(\frac{g}{2}D(\ch,\ch')+\frac{d}{2}u^2\right)(0,0,-1)+\mathbf{t}_{(H,\gamma,u)}(\ch, \ch').
\end{equation}
\end{lemma}
\begin{proof}
This is a direct computation by using (\ref{eq-w_sigma-decomposition}), (\ref{eq-relation1}), and (\ref{eq-relation2}).
\end{proof}

\subsection{The local Bayer--Macr\`i decomposition}
We decompose $w_\sigma$ into three cases according to the dimension of the support of objects with invariants $\ch$.
\emph{Assume} there is a flat family $\mathcal{E}\in \mathrm{D^b}(M_\sigma(\ch)\times S)$, and  denote the Mukai morphism by $\theta_{\sigma,\mathcal{E}}$.
\subsubsection{Supported in dimension $0$} Fix $\ch=(0,0,n)$, with $n$ a positive integer. Fix a frame $(H,\gamma,u)$. Since $t>0$ is the trivial chamber and there is no wall on $\Pi_{(H,\gamma,u)}$, we obtain $w_\sigma\xlongequal{\RR_+}\Im \mho_Z \xlongequal{\RR_+} \left(0,H,(\beta-\frac{3}{4}\KS).H\right)$, and the nef line bundle $\ell_\sigma=\theta_{\sigma,\mathcal{E}}(0,H,(sH-\frac{3}{4}\KS).H)$ on the moduli space $M_\sigma(\ch)\cong \mathrm{Sym}^n(S)$ (see \cite[Lemma 2.10]{LQ14}), which is independent of $s$.
\subsubsection{Supported in dimension $1$}\label{section-supp-dim1}
Fix a frame $(H,\gamma,u)$. We assume that $\ch=(0,\ch_1,\ch_2)$ with $\ch_1.H>0$. Now the center is given by $C=\frac{z+duy_2}{gy_1}$, which is independent of $\ch'$. 
So the vector
\begin{equation*}
\mathbf{t}_{(H,\gamma,u)}(\ch):=\mathbf{t}_{(H,\gamma,u)}(\ch, \ch')
\end{equation*}
is also independent of $\ch'$. There is another special vector 
\begin{equation*}
w_{\infty H,\beta} \xlongequal{\RR_+} (0,0,-1).
\end{equation*}
We get two well-defined line bundles in the following theorem:
\begin{equation}\label{eq-S-T}
\mathcal{S}:=\theta_{\sigma,\mathcal{E}}(0,0,-1),\quad \mathcal{T}_{(H,\gamma,u)}(\ch):=-\theta_{\sigma,\mathcal{E}}(\mathbf{t}_{(H,\gamma,u)}(\ch)).
\end{equation}
\begin{lemma} \label{lem-decomposition-dim1} (The local Bayer--Macr\`i decomposition in dimension $1$.) 
Fix a frame $(H,\gamma,u)$. Assume $\ch=(0,\ch_1,\ch_2)$ with $\ch_1.H>0$.
\begin{enumerate}
\item[(a)] There is a decomposition
\begin{equation*}\label{eq-w-sigma-general-dim1}
w_{\sigma\in W(\ch,\ch')} \xlongequal{\RR_+} \left(\frac{g}{2}D(\ch,\ch')+\frac{d}{2}u^2\right)(0,0,-1)+\mathbf{t}_{(H,\gamma,u)}(\ch),
\end{equation*}
where $(0,0,-1)$, $\mathbf{t}_{(H,\gamma,u)}(\ch)\in\mathbf{v}^\perp$. Moreover, $r=\ch_0'\neq 0$, and the coefficient before $(0,0,-1)$ is expressed in terms of the potential destabilizing Chern character $\ch'=(r,c_1 H+c_2 \gamma + \delta',\chi)$:
\begin{equation}\label{eq-w-sigma-1dim-simplify}
\frac{g}{2}D(\ch,\ch')+\frac{d}{2}u^2=\frac{\chi-gCc_1+udc_2}{r}.
\end{equation} 
\item[(b)]  Assume that there is a flat family $\mathcal{E}\in \mathrm{D^b}(M_\sigma(\ch)\times S)$. Then the Bayer--Macr\`i nef line bundle on the moduli space $M_\sigma(\ch)$ has a decomposition
\begin{equation}\label{eq-line-bundle-dim1}
\ell_{\sigma\in W(\ch,\ch')}\xlongequal{\RR_+}\left(\frac{g}{2}D(\ch,\ch')+\frac{d}{2}u^2\right)\mathcal{S}-\mathcal{T}_{(H,\gamma,u)}(\ch).
\end{equation}
\end{enumerate}
\end{lemma}
\begin{proof}
Part (a) follows from computation. The Mukai vector $\mathbf{v}$ is given by
$$\mathbf{v}=(0,\ch_1,\ch_2-\frac{1}{4}\ch_1.\KS).$$
So $(0,0,-1) \in \mathbf{v}^\perp$ by the definition in (\ref{eq-v-perpendicular}) and the formula (\ref{Mukaipairingsurface}). To show $\mathbf{t}_{(H,\gamma,u)}(\ch)\in \mathbf{v}^\perp$, we can either directly compute the Mukai pairing
$$\langle \mathbf{t}_{(H,\gamma,u)}(\ch), \mathbf{v}\rangle_S=(CH+u\gamma-\frac{3}{4}\KS)\cdot \ch_1-(\ch_2-\frac{1}{4}\ch_1.\KS-\frac{1}{2}\ch_1.\KS)=0$$
or note the relation (\ref{eq-w-sigma-general}) and the fact that $w_{\sigma\in W(\ch,\ch')}\in \mathbf{v}^\perp$, $(0,0,-1)\in \mathbf{v}^\perp$. Then $\mathcal{S}$ and $\mathcal{T}_{(H,\gamma,u)}(\ch)$ are well defined in (\ref{eq-S-T}).
Recall (\ref{eq-D_sigma}) for $D(\ch,\ch')$. Since $x=\ch_0=0$, we obtain $r\neq 0$. 
The relation (\ref{eq-w-sigma-1dim-simplify}) is then derived by using (\ref{eq-D-C-r}). Part (b) follows from part (a) by applying the Mukai morphism $\theta_{\sigma,\mathcal{E}}$.
\end{proof}

\subsubsection{Supported in dimension $2$} \label{section-supp-dim2}
Assume that $\ch_0\neq 0$. If $\ch_0<0$, we observe
\begin{equation*} \label{w-computation-DUW}
w_{\infty H,\beta} \xlongequal{\RR_+} (0,H,(\frac{\ch_1}{\ch_0}-\frac{3}{4}\KS)H) \xlongequal{\RR_+} w_{\sigma\in \mathtt{DUW}}.
\end{equation*}

\begin{definition}
Fix $\ch=(\ch_0,\ch_1,\ch_2)$ with $\ch_0\neq 0$, and define
\begin{eqnarray*}
\mathbf{w}(\ch)&:=&\left(1,-\frac{3}{4}\KS,-\frac{\ch_2}{\ch_0}-\frac{1}{2}\chi(\OO_S)+\frac{11}{32}\KS^2\right),\nonumber\\
\mathbf{m}(L,\ch)&:=& \left(0,L,(\frac{\ch_1}{\ch_0}-\frac{3}{4}\KS).L\right), \text{ where } L\in N^1(S),
\end{eqnarray*}
$$\mathbf{u}(\ch):=\mathbf{w}(\ch)+\mathbf{m}(\frac{1}{2}\KS,\ch)=\left(1,-\frac{1}{4}\KS,-\frac{\ch_2}{\ch_0}+\frac{\ch_1.\KS}{2\ch_0}-\frac{1}{2}\chi(\OO_S)-\frac{1}{32}\KS^2\right).$$
\end{definition}
\begin{lemma}\label{lem-perp-v} We have the following three perpendicular relations for Mukai vectors: 
$$\mathbf{m}(L,\ch), \ \mathbf{w}(\ch),\ \mathbf{u}(\ch)\in \mathbf{v}^\perp.$$
\end{lemma}
\begin{proof}
The perpendicular relations can be checked directly by (\ref{eq-v-perpendicular}), (\ref{Mukaivectorsurface}) and (\ref{Mukaipairingsurface}).
\end{proof}

\begin{lemma} \label{lem-decomposition-dim2} (The local Bayer--Macr\`i decomposition in dimension $2$.)
\begin{enumerate}
\item[(a)]If $\ch_0\neq 0$ and $\Im Z(\ch)>0$, then there is a decomposition (up to a positive scalar)
\begin{eqnarray}
w_{\omega,\beta}(\ch) &\xlongequal{\RR_+}&
\mu_{\sigma}(\ch)\mathbf{m}(\omega,\ch)+\mathbf{m}(\beta,\ch)+\mathbf{w}(\ch)\label{decomposition-w1}\\
&=& \mu_{\sigma}(\ch)\mathbf{m}(\omega,\ch)+\mathbf{m}(\alpha,\ch)+\mathbf{u}(\ch),\label{decomposition-w1-alpha}
\end{eqnarray}
where $\mathbf{m}(\omega,\ch)$, $\mathbf{m}(\beta,\ch)$, $\mathbf{m}(\alpha,\ch)$, $\mathbf{w}(\ch)$, $\mathbf{u}(\ch)\in \mathbf{v}^\perp$.
\item[(b)] Assume there is a flat family $\mathcal{E}$. 
Then the Bayer--Macr\`i line bundle class $\ell_{\sigma_{\omega,\beta}}$ has a decomposition in $N^1(M_\sigma(\ch))$:
\begin{equation}\label{decomposition-line-bundle1}
\ell_{\sigma_{\omega,\beta}}\xlongequal{\RR_+}\mu_{\sigma}(\ch)\theta_{\sigma,\mathcal{E}}(\mathbf{m}(\omega,\ch))+
\theta_{\sigma,\mathcal{E}}(\mathbf{m}(\beta,\ch))+
\theta_{\sigma,\mathcal{E}}(\mathbf{w}(\ch)).
\end{equation}
\end{enumerate}
\end{lemma}
\begin{proof}
Recall (\ref{eq-w_sigma-decomposition}). To show (\ref{decomposition-w1}), we only need to check that
\begin{equation}\label{eq-check1}
\frac{\ch_1}{\ch_0}.(\mu_\sigma(\ch) \omega+\beta)-\frac{\ch_2}{\ch_0}=\beta.(\mu_\sigma(\ch)\omega+\beta)-\frac{1}{2}(\omega^2+\beta^2).
\end{equation}
By the definition of the Bridgeland slope, we have
\begin{equation*}
\mu_{\sigma}(\ch)=\frac{\ch_2-\frac{1}{2}\ch_0 \left({\omega}^2-\beta^2\right)
-\ch_1.\beta}{\omega.\left(\ch_1-\ch_0\beta \right)}.
\end{equation*}
So
\begin{equation*}
\mu_{\sigma}(\ch)\omega.\left(\frac{\ch_1}{\ch_0}-\beta \right)=
\frac{\ch_2}{\ch_0}-\frac{1}{2}\left({\omega}^2-\beta^2\right)
-\frac{\ch_1}{\ch_0}.\beta.
\end{equation*}
Therefore, we have (\ref{eq-check1}). Then (\ref{decomposition-w1-alpha}) follows from (\ref{decomposition-w1}) and the relation $\alpha=\beta-\frac{1}{2}\KS$.
Part (b) follows directly by applying the Mukai morphism $\theta_{\sigma,\mathcal{E}}$.
\end{proof}

\begin{remark} \label{remark-BHL+}
An equivalent decomposition of $w_{\omega,\beta}(\ch)$ (\ref{decomposition-w1}) is also obtained by Bolognese, Huizenga, Lin, Riedl, Schmidt, Woolf, and Zhao \cite[Proposition 3.8.]{BHL+16}.
\end{remark}

\subsection{The global Bayer--Macr\`i decomposition} 
Assume the existence of the global Bayer--Macr\`i map. Assume that $W(\ch,\ch')$ is an actual Bridgeland wall. 
\subsubsection{Supported in dimension 1}
\begin{theorem} \label{thm-geometric-meaning-decomposition-dim1}  Fix a frame $(H,\gamma,u)$ and assume $\ch=(0,\ch_1,\ch_2)$ with $\ch_1.H>0$. Assume the existence of the global Bayer--Macr\`i map, with the fixed base stability condition in the Simpson chamber $\mathtt{SC}$, that is, $M_{\sigma\in\mathtt{SC}}(\ch)\cong M_{(\alpha, \omega)}(\ch)$. Then there is a correspondence from Bridgeland wall $W(\ch,\ch')$ as semicircle (\ref{eq-wall-equation}) in the half-plane $\Pi_{(H,\gamma,u)}$ with fixed center $C$  (or equivalently, as semiline (\ref{eq-wall-parallel-lines}) in the plane $\Sigma_{(H,\gamma,u)}$ with fixed slope $C$) to the Mori wall inside pseudoeffective cone $\overline{\mathrm{Eff}}(M_{(\alpha, \omega))}(\ch)$:
\begin{eqnarray}
\ell_{\sigma\in W(\ch,\ch')}&\xlongequal{\RR_+}&\left(\frac{g}{2}D(\ch,\ch')+\frac{d}{2}u^2\right)\mathcal{S}-\mathcal{T}_{(H,\gamma,u)}(\ch)\label{eq-line-bundle-dim1-decomposition}\\
&=& \frac{\chi-gCc_1+udc_2}{r}\mathcal{S}-\mathcal{T}_{(H,\gamma,u)}(\ch).\label{eq-w-sigma-1dim-simplify-line-bundle}
\end{eqnarray}
\end{theorem}
\begin{proof}
The center $C=\frac{z+duy_2}{gy_1}$ is independent of $\ch'$. So is $\mathcal{T}_{(H,\gamma,u)}(\ch)$. The number $D(\ch,\ch')$ is given by (\ref{eq-D-C-r-F}). Since the base stability condition $\sigma\in \mathtt{SC}$, we obtain 
$$\mathcal{S}=\theta_{\mathtt{SC},\mathcal{E}}((0,0,-1))\in N^1(M_{(\alpha,\omega)}(\ch)),$$
$$\mathcal{T}_{(H,\gamma,u)}(\ch)=-\theta_{\mathtt{SC},\mathcal{E}}(\mathbf{t}_{(H,\gamma,u)}(\ch))\in N^1(M_{(\alpha,\omega)}(\ch)).$$ Then (\ref{eq-line-bundle-dim1-decomposition}) follows from (\ref{eq-line-bundle-dim1}) by fixing the above two line bundles $\mathcal{S}$ and $\mathcal{T}_{(H,\gamma,u)}(\ch)$ in the Simpson moduli space. We obtain (\ref{eq-w-sigma-1dim-simplify-line-bundle}) by using (\ref{eq-w-sigma-1dim-simplify}).
\end{proof}

By using the Donaldson morphism $\lambda_\mathcal{E}$, we have \cite[Example 8.1.3]{HL10} 
$$\mathcal{S}=\lambda_\mathcal{E}(0,0,1)=(p_M)_*(\mathrm{det}(\mathcal{E})|_{M\times \{s\}}).$$
The line bundle $\mathcal{S}$ is conjectured to induce the support morphism, which maps $E\in M_{\sigma\in\mathtt{SC}}(\ch)$ to $\Supp(E)$. This is proved for the case in which $S=\mathbb{P}^2$ (see \cite{Woo13}) or $S$ is a K3 surface (see \cite[Lemma 11.3]{BM14b}).

\subsubsection{Supported in dimension $2$} By applying the derived dual functor if necessary, we further assume $\ch_0>0$ in Lemma~\ref{lem-decomposition-dim2}.
Recall Lemma~\ref{lem-duality-line-bundle} and Corollary \ref{cor-derived-dual-stability}. We obtain
\begin{eqnarray*}
w_{\omega,-\beta}(-\ch^*) &\xlongequal{\RR_+}&
\mu_{\Phi(\sigma)}(-\ch^*)\mathbf{m}(\omega,-\ch^*)+\mathbf{m}(-\beta,\ch)+\mathbf{w}(-\ch^*),\label{decomposition-w1-dual}\\
\ell_\sigma\cong\ell_{\Phi(\sigma)}&\xlongequal{\RR_+}&
\mu_{\Phi(\sigma)}(-\ch^*)\theta_{\Phi(\sigma),\mathcal{F}}(\mathbf{m}(\omega,-\ch^*))\\
& &+\theta_{\Phi(\sigma),\mathcal{F}}(\mathbf{m}(-\beta,-\ch^*))+
\theta_{\Phi(\sigma),\mathcal{F}}(\mathbf{w}(-\ch^*)).
\end{eqnarray*}
Since $\mu_{\Phi(\sigma)}(-\ch^*)=-\mu_{\sigma}(\ch)$, we get 
$$\theta_{\Phi(\sigma),\mathcal{F}}(\mathbf{w}(-\ch^*))\cong\theta_{\sigma,\mathcal{E}}(\mathbf{w}(\ch)),\quad \theta_{\Phi(\sigma),\mathcal{F}}(\mathbf{m}(\omega,-\ch^*))\cong -\theta_{\sigma,\mathcal{E}}(\mathbf{m}(\omega,\ch)).$$

\begin{notation}\label{notation-supp-dim2} Assume $\ch_0>0$. Let $L$ be a line bundle on $S$. Denote
$$\widetilde{L}:=\theta_{\Phi(\sigma),\mathcal{F}}(\mathbf{m}(L,-\ch^*))\cong -\theta_{\sigma,\mathcal{E}}(\mathbf{m}(L,\ch)),\quad
\mathcal{B}_0:=-\theta_{\sigma,\mathcal{E}}(\mathbf{u}(\ch)).$$
Then we have
\begin{equation}\label{eq-B0}
\theta_{\sigma,\mathcal{E}}(\mathbf{w}(\ch))=\frac{1}{2}\widetilde{\KS}-\mathcal{B}_0.
\end{equation}
Recall $\alpha=\beta-\frac{1}{2}\KS$. Denote 
\begin{equation*}
\mathcal{B}_{\alpha}:=\widetilde{\beta}-\theta_{\Phi(\sigma),\mathcal{F}}(\mathbf{w}(-\ch^*))
\cong \widetilde{\beta}-\theta_{\sigma,\mathcal{E}}(\mathbf{w}(\ch))=\widetilde{\alpha}+\mathcal{B}_0.
\end{equation*}
Note that $\widetilde{L}$, $\mathcal{B}_{\alpha}$, and $\mathcal{B}_{0}$ are line bundles on $M_{\sigma}(\ch)$.
\end{notation}
\begin{assumption}\label{assumption-ch}
The Chern character $\ch=(\ch_0,\ch_1,\ch_2)$ satisfies \emph{condition} $\mathrm{(C)}$ if the following three assumptions hold:
\begin{itemize} 
\item $\ch_0> 0$;
\item (Bogomolov type) $\ch_1^2-2\ch_0\ch_2\geq 0$; and
\item $\mathrm{gcd}(\ch_0,\ch_1.H, \ch_2-\frac{1}{2}\ch_1.\KS)=1$ for a fixed ample line bundle $H$ (see \cite[Corollary 4.6.7]{HL10}).
\end{itemize}
\end{assumption}
\begin{theorem}\label{thm-geometric-meaning-decomposition-dim2}
Fix $\ch$, and assume it satisfies condition $\mathrm{(C)}$. Assume the existence of the global Bayer--Macr\`i map, with the fixed base stability condition in the Gieseker chamber $\mathtt{GC}$, that is, $M_{\sigma\in\mathtt{GC}}(\ch)\cong M_{(\alpha, \omega)}(\ch)$. Then the following conclusions hold.
\begin{enumerate}
\item[(a)] There is a Bayer--Macr\`i decomposition for the line bundle $\ell_{\sigma_{\omega,\beta}}$:
\begin{equation*}\label{eq-geometric-meaning-decomposition}
\ell_{\sigma_{\omega,\beta}}\xlongequal{\RR_+} (-\mu_{\sigma_{\omega,\beta} }(\ch))\widetilde{\omega}-\mathcal{B}_{\alpha}=(-\mu_{\sigma_{\omega,\beta}}(\ch))\widetilde{\omega}-\widetilde{\alpha}-\mathcal{B}_0.
\end{equation*}
\item[(b)] The line bundle $\widetilde{\omega}$ induces the Gieseker--Uhlenbeck morphism from the $(\alpha,\omega)$-Gieseker semistable moduli space $M_{(\alpha,\omega)}(\ch)$ to the Uhlenbeck space $U_\omega (\ch)$. 
\item[(c)] If $\ch_0=2$ and $\partial M_{(\alpha,\omega)}(\ch)\neq\emptyset$, the divisor $\mathcal{B}_{\alpha}$ is the $\alpha$-twisted boundary divisor of the induced Gieseker--Uhlenbeck morphism. In particular, in the case of $\alpha=0$, the divisor $\mathcal{B}_0$ is the (untwisted) boundary divisor from the $\omega$-semistable Gieseker moduli space $M_\omega(\ch)$ to the Uhlenbeck space $U_\omega (\ch)$. 
\item[(d)] Fix a frame $(H,\gamma,u)$. Then there is a correspondence from the Bridgeland wall $W(\ch,\ch')$ as semicircle (\ref{eq-wall-equation}) in the half-plane $\Pi_{(H,\gamma,u)}$ with center $C(\ch,\ch')$  (or equivalently, as semiline (\ref{eq-wall-passing-fixed-point}) in the plane $\Sigma_{(H,\gamma,u)}$ with slope $C(\ch,\ch')$) to the effective line bundle on the moduli space $M_{(\alpha, \omega)}(\ch)$:
\begin{equation} \label{eq-center-linebundle}
\ell_{\sigma\in W(\ch,\ch')}\xlongequal{\RR_+}
-C(\ch,\ch')\widetilde{H}-u\widetilde{\gamma}+\frac{1}{2}\widetilde{\KS}-\mathcal{B}_0.
\end{equation}
\end{enumerate}
\end{theorem}
\begin{proof}
We identify $\widetilde{\omega}$ and $\mathcal{B}_{\alpha}$ as line bundles on $M_{(\alpha, \omega)}(\ch)$. Then (\ref{decomposition-line-bundle1}) implies part (a). Since $-\mu_{\sigma\in \mathtt{UW}}(\ch)=\mu_{\Phi(\sigma)\in\mathtt{DUW}}(-\ch^*)=+\infty$, we obtain
$$ \ell_{\sigma\in\mathtt{UW}}\xlongequal{\RR_+} \widetilde{\omega}.
$$
Part (b) follows from  \cite[Theorem 8.2.8]{HL10}. 
If $\ch_0=2$ and $\partial M_{(\alpha,\omega)}(\ch)\neq\emptyset$, then the Gieseker--Uhlenbeck morphism is a divisorial contraction by \cite[Lemma 9.2.1]{HL10}, and $\mathcal{B}_\alpha$ is the boundary divisor from the $\alpha$-twisted moduli space $M_{(\alpha,\omega)}(\ch)$ to the Uhlenbeck space $U_\omega(\ch)$. In particular, $\mathcal{B}_0$ is the untwisted boundary divisor. This shows part (c). Then (\ref{eq-center-linebundle}) follows from a direct computation by by using (\ref{decomposition-line-bundle1}), (\ref{lem-direct-relations-using-centers}), and (\ref{eq-B0}):
\begin{eqnarray*}
\ell_{\sigma\in W(\ch,\ch')}&\xlongequal{\RR_+}&\theta_{\sigma,\mathcal{E}}(\mathbf{m}(\mu_{\sigma}(\ch)\omega+\beta,\ch))+
\theta_{\sigma,\mathcal{E}}(\mathbf{w}(\ch))\\
&=&-C(\ch,\ch')\widetilde{H}-u\widetilde{\gamma}+\frac{1}{2}\widetilde{\KS}-\mathcal{B}_0. 
\end{eqnarray*}
\end{proof}

There are two line bundles $\mathcal{L}_0$, $\mathcal{L}_1$ on Gieseker moduli space introduced by Le Potier. We follow the notation from \cite[Definition 8.1.9.]{HL10}.
Then \begin{equation*}\label{eq-Potier-BM}
\mathcal{L}_0=-\ch_0 \mathcal{B}_0,\quad \mathcal{L}_1=\ch_0 \widetilde{H}.
\end{equation*} 

Arcara, Bertram, Coskun, and Huizenga \cite{ABCH13} studied the Hilbert scheme of $n$-points on the projective plane $\mathbb{P}^2$ and gave a precise conjecture between the Bridgeland walls and Mori walls, which was one of the motivations of Bayer--Macr\`i's line bundle theory. This conjecture was proved by Li and Zhao \cite[Theorem 1.2]{LZ18}. The relation still holds for a more general primitive character (see \cite{LZ16}).   
\begin{corollary}[\cite{LZ16}]\label{cor-ABCH}
Let $S=\mathbb{P}^2$, and denote by $H$ the hyperplane divisor on $\mathbb{P}^2$. Fix $\ch$ primitive with $\ch_0>0$. Assume $M_{H}(\ch)\neq \emptyset$. Then there is a relation
\begin{equation*}
\ell_{\sigma\in W(\ch,\ch')}\xlongequal{\RR_+}
-\left(C(\ch,\ch')+\frac{3}{2}\right)\widetilde{H}-\mathcal{B}_0.
\end{equation*}
\end{corollary}
\begin{proof}
The existence of the global Bayer--Macr\`i map was proved by Li and Zhao \cite[Theorem 0.2]{LZ16}. We then apply (\ref{eq-center-linebundle}) with $\gamma=0$ and $\KS=-3H$.
\end{proof}

\begin{example} \label{boundary-divisor-Hilbert-schemes}
Assume that the irregularity of the surface is $0$. If $\ch=(1,0,-n)$, then the Gieseker--Uhlenbeck morphism is the Hilbert--Chow morphism $h: S^{[n]}\to S^{(n)}$, which maps the Hilbert scheme of $n$-points on $S$ to the symmetric product $S^{(n)}$. In particular, 
$$\widetilde{H}=\mathcal{L}_1= h^*(\OO_{S^{(n)}}(1)),$$
which induces the Hilbert--Chow morphism (see \cite[Example 8.2.9]{HL10}). The boundary divisor of the Hilbert--Chow morphism is
\begin{equation} \label{eq-HC-Boundary-divisor}
\mathrm{B}:=\{\xi \in S^{[n]}: |\Supp(\xi)|<n\}.
\end{equation}
It is known from \cite[Appendix]{BSG91} that $\frac{1}{2}\mathrm{B}$ is an integral divisor and $\mathcal{B}_0=-\mathcal{L}_0=\frac{1}{2}\mathrm{B}$.
\end{example}

\section{A toy model: fibered surface over $\mathbb{P}^1$ with a global section}\label{section-toy-model} We compute the nef cone of the Hilbert scheme $S^{[n]}$ of $n$-points over $S$ by using Theorem $\ref{thm-geometric-meaning-decomposition-dim2}$. Here let $\pi:S\to \mathbb{P}^1$ be either a $\mathbb{P}^1$-fibered or an elliptic-fibered surface over $\mathbb{P}^1$ with a global section $E$ whose self-intersection number is $-e$. We \emph{assume} that  all fibers are reduced and irreducible, and the Picard group of $S$ is generated by $E$ and $F$, where $F$ is the generic fiber class. We have the intersection numbers: 
$$E.E=-e, \quad E.F=1, \quad F^2=0.$$

\begin{itemize}
\item $\mathbb{P}^1$ fibration. In this case, $F\cong \mathbb{P}^1$ and $S$ is the \emph{Hirzebruch surface} $\Sigma_e$ with integer $e\geq 0$. Then $\KS=-2(E+eF)+(e-2)F$. Here $\Sigma_0$ is the surface $\mathbb{P}^1\times\mathbb{P}^1$.
\item Elliptic fibration. In this case, the generic fiber $F$ is an elliptic curve. We denote the surface by $S_e$ and further \emph{assume} that $e\geq 2$. Then $S_e$ has the unique section $E$ and $\KS=(e-2)F$ (see \cite{Mir89}). 
\end{itemize}

Since the nef cone of $S$ is generated by the two extremal nef line bundles $E+eF$ and $F$, any ample line bundle $H$, after rescaling, can be written as
$$H:=\lambda(E+eF)+(1-\lambda)F, \quad 0<\lambda<1.$$
Take $\gamma$ such that $H.\gamma=0$ and $H^2=-\gamma^2$. Basic computation shows that $\gamma=\pm \left(-\lambda(E+eF)+(1-\lambda+e\lambda)F\right)$. An $(H,\gamma,u)$-frame, with $u\geq 0$, is fixed by the choice
$$\gamma:=-\lambda(E+eF)+(1-\lambda+e\lambda)F.$$
The two numbers $\lambda$ and $u$ are regarded as the \emph{initial values}. 

Fix $\ch=(1,0,-n)$, with integer $n\geq 2$. The potential walls are given by $(s-C)^2+t^2=C^2+D$ with $t>0$,  where 
\begin{equation*}
C=C(\ch,\ch')=\frac{\ch_2'+\ch_0' n-u\ \ch_1'.\gamma}{\ch_1'.H},\quad D=D(\ch,\ch')=-u^2-\frac{2n}{H^2}.
\end{equation*}
Recall $s_0:=\frac{\ch_1.H}{\ch_0 H^2}=0$. The $\mathtt{UW}$ is given by $s=s_0=0$. Therefore, $C<0$. 

One type of nef line bundle on $S^{[n]}$ is $\widetilde{\omega}$ for $\omega\in\amp(S)$, which induces a Gieseker--Uhlenbeck morphism. By taking $\omega$ to be extremal, that is, $\omega=E+eF$ or $F$, we obtain two extremal nef line bundles  on $S^{[n]}$:
\begin{equation}\label{eq-two-nef-boundaries}
\widetilde{(E+eF)},\quad \widetilde{F}.
\end{equation}

To find the nef cone of $S^{[n]}$, we need to find the biggest nontrivial wall, that is, the smallest value of $C$. Let us call such a wall the \emph{Gieseker wall}.
 
\begin{lemma}\label{lemma_two-Gieseker-walls}
If the Gieseker wall is given by a rank $1$ wall, then
$$\ch'=(1,-F,0) \quad \text{or}\quad \ch'=(1,-E,\frac{-e}{2}).$$
\end{lemma} 
\begin{proof}
The idea is the same as in \cite{ABCH13} (see also \cite{BC13}). Any destabilizing subsheaf of $I_Z$ of rank $1$ has the form $L\otimes I_W$ with Chern character $\ch'=(1,L,\frac{L^2}{2}-w)$ for some line bundle $L$ and some ideal sheaf $I_W$ of length $w\geq 0$. Then $C(\ch,\ch')=\frac{\frac{L^2}{2}+n-u L.\gamma}{L.H}+\frac{w}{-L.H}$. To guarantee $L\otimes I_W$ as an object in the heart, we need $L.H<0$. Then we write $L=-(mF+kE)$, with two nonnegative integers $m$ and $k$, and $(m,k)\neq (0,0)$. To get the biggest nontrivial wall, we must take $w=0$. Denote the line bundle on $S^{[n]}$ corresponding to the destabilizing line bundle $-(mF+kE)$ by $\ell(m,k)$. The locus contracted by $\ell(0,1)$ is $\{Z\in S^{[n]}|\ Z\subset E\}$. The locus contracted by $\ell(1,0)$ is $\{Z\in S^{[n]}|\ Z\subset F,\ Z \text{ is linear equivalent to } n (E\cap F)\}$. Assume that the smallest value is obtained by taking $(m,k)\neq (1,0)$ nor $(0,1)$. Recall that the walls are nested. But the loci contracted by $\ell(1,0)$ or $\ell(0,1)$ are also contracted by $\ell(m,k)$, which is a contradiction.
\end{proof}

The line bundles $\ell(1,0)$ and $\ell(0,1)$ depend on the initial values $\lambda$ and $u$. By (\ref{eq-center-linebundle}), we have
\begin{eqnarray*}
\ell(0,1)_{\lambda,u} &=& n\widetilde{(E+eF)}+\left((\frac{1-\lambda}{\lambda})n-u\left(2(1-\lambda)+e\lambda\right)\right)\widetilde{F}+\frac{1}{2}\widetilde{\KS}-\mathcal{B}_0;\\
\ell(1,0)_{\lambda,u} &=& \left((n-\frac{1}{2}e)(\frac{\lambda}{1-\lambda})+u(\lambda+(\frac{\lambda}{1-\lambda})(e\lambda+1-\lambda))\right)\widetilde{(E+eF)}\\
& &+\left(n-\frac{1}{2}e\right)\widetilde{F}+\frac{1}{2}\widetilde{\KS}-\mathcal{B}_0.
\end{eqnarray*}
By taking $\lambda$ to $0^+$ or $1^-$, respectively, we get two nef boundaries as in (\ref{eq-two-nef-boundaries}): 
\begin{equation*}
\ell(0,1)_{0^+,u}\xlongequal{\RR_+} \widetilde{F},\quad
\ell(1,0)_{1^-,u}\xlongequal{\RR_+} \widetilde{(E+eF)}.
\end{equation*}
Moreover, $\ell(0,1)_{\lambda,u}$ is decreasing and $\ell(1,0)_{\lambda,u}$ is increasing with respect to $\lambda$. The two types of loci are simultaneously contracted if and only if 
\begin{equation*}
\ell(0,1)_{\lambda,u}=\ell(1,0)_{\lambda,u}.
\end{equation*}
The solutions are given by
\begin{equation}\label{eq-solution}
u=U(\ch, \lambda):=\frac{\frac{n}{\lambda}-\frac{n-\frac{e}{2}}{1-\lambda}}{2+e\cdot\frac{\lambda}{1-\lambda}} \text{ and } 0<\lambda<1.
\end{equation}
Moreover, with $0<\lambda<1$, we have
\begin{equation}\label{eq-line_bundle}
\ell(0,1)_{\lambda,U(\ch, \lambda)}=\ell(1,0)_{\lambda,U(\ch, \lambda)}=
n\widetilde{(E+eF)}+(n-\frac{e}{2})\widetilde{F}+\frac{1}{2}\widetilde{\KS}-\mathcal{B}_0.
\end{equation}

\begin{theorem}\label{thm_nef-cone}
\begin{enumerate}
\item[(a)] \cite[Theorem 1(2), 1(3)]{BC13} The nef cone $\nef(\Sigma_e^{[n]})$ ($e\geq 0$, $n\geq 2$) is generated by the nonnegative combinations of
$\widetilde{(E+eF)}$, $\widetilde{F}$, and $(n-1)\widetilde{(E+eF)}+(n-1)\widetilde{F}-\frac{1}{2}\mathrm{B}$.
\item[(b)] The nef cone $\nef(S_e^{[n]})$ ($e\geq 2$, $n\geq 2$) is generated by the nonnegative combinations of
$\widetilde{(E+eF)}$, $\widetilde{F}$, and $n\widetilde{(E+eF)}+(n-1)\widetilde{F}-\frac{1}{2}\mathrm{B}$.
\end{enumerate}
\end{theorem}
\begin{proof}
We only need to show that the Gieseker wall is not a higher rank wall in the case of (\ref{eq-solution}). Note that the line bundle in (\ref{eq-line_bundle}) is independent of the $\lambda$. So it is enough to check for the case $\lambda=\frac{1}{2}$. Now we have $u=U(\ch,\frac{1}{2})=\frac{e}{e+2}$. The two walls given by Lemma~\ref{lemma_two-Gieseker-walls} coincide with the center $C=-2n+\frac{e}{e+2}$. By the estimation formula from \cite[Section 5]{BC13}, the center $C_{k}$ of a rank $k$ wall ($k\geq 2$) is bounded by 
\begin{equation*}
C_{k}^2\leq \left(u^2+\frac{2n}{H^2}\right)\frac{(2k-1)^2}{(2k-1)^2-1}\leq \left(u^2+\frac{2n}{H^2}\right)\frac{9}{8}.
\end{equation*}
Now $H^2=\frac{e}{4u}$, and $u=\frac{e}{e+2}$. It is easy to check that
\begin{equation*}
\left(u^2+\frac{8nu}{e}\right)\frac{9}{8}<\left(-2n+u\right)^2.
\end{equation*}
So $C_{k}^2<C^2$. Therefore, higher rank walls are strictly inside the wall given by center $C$. By (\ref{eq-center-linebundle}) and equation~(\ref{eq-line_bundle}), the extremal nef line bundle corresponding to $C$ is 
\begin{equation*}\label{eq-third-nef}
n\widetilde{(E+eF)}+(n-\frac{e}{2})\widetilde{F}+\frac{1}{2}\widetilde{\KS}-\mathcal{B}_0.
\end{equation*}
The nef cone of $S^{[n]}$ is generated by the nonnegative combinations of 
\begin{equation*}
\widetilde{(E+eF)},\quad \widetilde{F},\quad n\widetilde{(E+eF)}+(n-\frac{e}{2})\widetilde{F}+\frac{1}{2}\widetilde{\KS}-\mathcal{B}_0.
\end{equation*} 
Recall that $\mathcal{B}_0=\frac{1}{2}\mathrm{B}$ in Example~\ref{boundary-divisor-Hilbert-schemes}. The proof is completed by using $\KS=-2(E+eF)+(e-2)F$ or $(e-2)F$, respectively.
\end{proof}

The above computation suggests that the number $u$ plays an important role in order to find the extremal nef line bundle, and in general $u\neq 0$, that is, $\omega$ is not parallel to $\beta$. The nef cone of $\Sigma_e^{[n]}$ has been obtained by Bertram and Coskun \cite{BC13}. The nef cone of $S_2^{[n]}$ ($n\geq 2$) has been obtained by J. Li and W.-P. Li \cite{LL10}. Both of the results use the notion of $k$-very ample line bundles (see \cite{BSG91}).

\appendix

\section{Twisted Gieseker stability and the large-volume limit}\label{appendix-large-volume-limit}
\begin{definition}\label{def-twisted-Gieseker}  (\cite[Definition 3.4]{EG95}, \cite[Definition 4.1]{FQ95}, \cite[Definition 3.2]{MW97})
Let $\omega,\alpha\in \NS(S)_{\QQ}$ with $\omega$ ample. For $E\in \Coh(S)$, we denote the leading coefficient of $\chi (E\otimes \alpha^{-1}\otimes \omega^{\otimes m})$ with respect to $m$ by $a_d$. A coherent sheaf $E$ of dimension $d$ is said to be $\alpha$-twisted $\omega$-Gieseker-(semi)stable  if $E$ is pure and for all $0\neq F\subsetneq (\subseteq) E$,
\begin{equation}\label{twisted}
\frac{\chi (F\otimes \alpha^{-1}\otimes \omega^{\otimes m})}{a_d(F)}<(\leq)
\frac{\chi (E\otimes \alpha^{-1}\otimes \omega^{\otimes m})}{a_d(E)}
\quad\text{for } m\gg 0.
\end{equation}
We also write the $\alpha$-twisted $\omega$-Gieseker-(semi)stability as the $(\alpha,\omega)$-Gieseker (semi)stability. Denote $M_{(\alpha,\omega)}(\ch)$ (if it exists) as the moduli space of $(\alpha,\omega)$-semistable sheaves $E$ with $\ch(E)=\ch$.
\end{definition}

\begin{proposition-definition}\cite[Theorem 1.2]{LQ14} \label{def-Gieseker-chamber}
Fix $\ch=(\ch_0,\ch_1,\ch_2)$. Fix a frame $(H,\gamma,u)$, and consider 
$\sigma_{\omega,\beta}$ on the $(s,t)$-half-plane $\Pi_{(H,\gamma,u)}$ (Definition~\ref{def-coordinate}). Denote $s_0:=\frac{\ch_1.H}{\ch_0 H^2}$ if $\ch_0\neq 0$. We always fix the relation
\begin{equation}\label{alpha1}
\alpha=\beta-\frac{1}{2}\KS.
\end{equation}
\begin{itemize}
\item[$\mathtt{TC}$] If $\ch=(0,0,n)$ with positive integer $n$, then there is \emph{no wall}, and $t>0$ is the \emph{trivial chamber} in $\Pi_{(H,\gamma,u)}$. And
$$M_{\sigma_{\omega,\beta}}(\ch)=M_{(\alpha,\omega)}(\ch)=\mathrm{Sym}^n(S). $$ 
\item[$\mathtt{SC}$] If $\ch_0=0$, and $\ch_1.H>0$, we define the chamber for $t\gg 0$ as the \emph{Simpson chamber with respect to $(H,\gamma,u)$}. Then $$M_{\sigma_{\omega,\beta}\in \mathtt{SC}}(\ch)=M_{(\alpha,\omega)}(\ch).$$
And the $(\alpha,\omega)$-Gieseker semistability is the Simpson semistability defined by the slope $\frac{\ch_2(E)-\ch_1(E).\beta}{\omega.\ch_1(E) }$.
\item[$\mathtt{GC}$] If $\ch_0>0$, we define the chamber for $t\gg 0$ and $s<s_0$ as the \emph{Gieseker chamber with respect to $(H,\gamma,u)$}. If $\ch$ satisfies condition $\mathrm{(C)}$, then
$$M_{\sigma_{\omega,\beta}\in \mathtt{GC}}(\ch)\cong M_{(\alpha,\omega)}(\ch).$$
\item[$\mathtt{UW}$] If $\ch_0>0$, we define the wall $t>0$ and $s=s_0$, that is, $\Im Z(\ch)=0$ as the \emph{Uhlenbeck wall with respect to $(H,\gamma,u)$}.
\item[$\mathtt{DGC}$] If $\ch_0<0$, we define the chamber for $t\gg 0$ and $s>s_0$ as the \emph{dual Gieseker chamber with respect to $(H,\gamma,u)$}. If $-(\ch)^*$ satisfies condition $\mathrm{(C)}$, then by Lemma~\ref{lem-duality},
$$M_{\sigma_{\omega,\beta}\in \mathtt{DGC}}(\ch)\cong M_{\sigma_{\omega,-\beta}\in \mathtt{GC}}(-(\ch)^*).$$
\item[$\mathtt{DUW}$] If $\ch_0<0$, we define the wall $t>0$ and $s=s_0$, that is, $\Im Z(\ch)=0$ as the \emph{dual Uhlenbeck wall with respect to $(H,\gamma,u)$}. If $-(\ch)^*$ satisfies condition $\mathrm{(C)}$, then
$$M_{\sigma_{\omega,\beta}\in \mathtt{DUW}}(\ch)\cong U_{\omega}(-(\ch)^*),$$
where $U_{\omega}(-(\ch)^*)=U_{H}(-(\ch)^*)$ is the Uhlenbeck compactification of the moduli space $M^{\mathrm{lf}}_{\omega}(-(\ch)^*)$ of locally free sheaves with invariant $-(\ch)^*$.
\end{itemize}
\end{proposition-definition}

\section{Bayer--Macr\`i decomposition on K3 surfaces by using $\hat{\sigma}_{\omega,\beta}$}\label{appendix-K3}

Let $S$ be a smooth projective surface. By some physical hints (e.g., \cite[Section 6.2.3]{Asp05}), the central charge is often taken as (e.g., \cite{Bri08,BM14a})
\begin{equation}\label{eq-Z-twisted}
\hat{Z}_{\omega,\beta}(E):=-\int_{S} e^{-(\beta+\sqrt{-1}\omega)}.\ch(E).\sqrt{\td{(S)}}.
\end{equation}
Similarly to Lemma~\ref{lem-central-charge}, one can check that
\begin{equation}
\hat{Z}_{\omega,\beta}(E)= \langle \mho_{\hat{Z}}, v(E) \rangle_S, \quad\text{where}\quad
\mho_{\hat{Z}_{\omega,\beta}} := e^{\beta-\frac{1}{2}\KS+\sqrt{-1}\omega}. \label{Omega-Z-hat}
\end{equation}
Write $\mho_{\hat{Z}_{\omega,\beta}}$ as $\mho_{\hat{Z}}$. Basic computation shows that
\begin{equation*}
\langle \mho_Z, \mho_Z \rangle_S = \chi(\OO_S)-\frac{1}{4}\KS^2,\quad \langle \mho_{\hat{Z}}, \mho_{\hat{Z}} \rangle_S = -\frac{1}{8}\KS^2.
\end{equation*}
Recall from \cite{Bri09} and \cite{BB17} that a numerical stability condition $\sigma$ is called \emph{reduced} if the corresponding $\pi(\sigma)$ satisfies $\langle \pi(\sigma), \pi(\sigma) \rangle_S =0$.

In the following, we always \emph{assume} that $S$ is a smooth projective K3 surface, and \emph{assume} that $\hat{Z}_{\omega,\beta}(F)\notin \RR_{\leq 0}$ for all spherical sheaves $F\in\Coh(S)$. Then $\hat{\sigma}_{\omega,\beta}=(\hat{Z}_{\omega,\beta}, \A_{\omega,\beta})$ is a reduced numerical geometric Bridgeland stability condition (see \cite[Lemma 6.2]{Bri08}). Let $\mathbf{v}=v(\ch)\in H^*_{\mathrm{alg}}(S,\ZZ)$ be a primitive class with $\langle\mathbf{v},\mathbf{v}\rangle_S>0$. Define $\hat{w}_{\omega,\beta}:=\hat{w}_{\hat{\sigma}}:=-\Im \left(\overline{\langle \mho_{\hat{Z}}, \mathbf{v} \rangle_S}\cdot \mho_{\hat{Z}}\right)$. Define $\ell_{\hat{\sigma},\mathcal{E}}$ similarly to that in (\ref{Bayer-Macri-def}) but use $\hat{Z}$ instead. Then 
$$\ell_{\hat{\sigma}_{\omega,\beta}}\xlongequal{\RR_+}\theta_{\hat{\sigma},\mathcal{E}}(\hat{w}_{\omega,\beta}).$$
Fix a frame $(H,\gamma,u)$. The potential walls $\hat{W}(\ch,\ch')$ in the $(s,t)$-model
are given by semicircles (or in the $(s,q)$-model are given by semilines)
\begin{equation}\label{eq-wall-K3}
(s-C)^2+t^2=C^2+D+\frac{2}{H^2} \quad (\text{or } q=Cs+\frac{1}{2}D+\frac{1}{H^2}),
\end{equation}
where $C$ and $D$ are defined in Theorem~\ref{thm-Maciocia}. 
There is a global Bayer--Macr\`i map (see \cite[Theorem 1.2]{BM14b}).

\begin{theorem}\label{thm-K3-BM-decomposition} (Bayer--Macr\`i decomposition on K3 surfaces.)
Use the notation and assumptions as above.
\begin{itemize}
\item If $\ch_0=0$ and $\ch_1.H>0$, then the Bayer--Macr\`i line bundle has a decomposition
\begin{equation}\label{eq-line-bundle-dim1-hat}
\ell_{\hat{\sigma}\in \hat{W}(\ch,\ch')}\xlongequal{\RR_+}\left(\frac{g}{2}D(\ch,\ch')+\frac{d}{2}u^2\right)\mathcal{S}-\mathcal{T}_{(H,\gamma,u)}(\ch).
\end{equation}
The line bundle $\mathcal{S}$ induces the support morphism.
\item If $\ch_0>0$, then the Bayer--Macr\`i line bundle has a decomposition
\begin{equation}
\ell_{\hat{\sigma}\in \hat{W}(\ch,\ch')}\xlongequal{\RR_+}-C\widetilde{H}-u\widetilde{\gamma}-\mathcal{B}_0.
\end{equation}
The line bundle $\widetilde{\omega}$ (or $\widetilde{H}$) induces the Gieseker--Uhlenbeck morphism. 
\end{itemize}
\end{theorem}

\end{document}